\documentclass{amsart}
\usepackage[all]{xy}
\usepackage{amsmath,amsthm,amsfonts,amssymb,amscd,amsbsy,dsfont,multirow,hyperref}

\hypersetup{
    pdftoolbar=true,
    pdfmenubar=true,
    pdffitwindow=false,
    pdfstartview={FitH},
    pdftitle={Deforming solutions of geometric variational problems with varying symmetry groups},
    pdfauthor={Renato G. Bettiol, Paolo Piccione and Gaetano Siciliano},
    pdfsubject={46T05, 47J07, 58C15, 58D19},
    pdfkeywords={}
    pdfnewwindow=true,
    colorlinks=true, 
    linkcolor=blue,
    citecolor=blue,
    urlcolor=blue,
}

\author[R. G. Bettiol]{Renato G. Bettiol}
\author[P. Piccione]{Paolo Piccione}
\author[G. Siciliano]{Gaetano Siciliano}

\address{
\begin{tabular}{lll}
University of Notre Dame & &Universidade de S\~ao Paulo \\
Department of Mathematics & & Departamento de Matem\'atica \\
255 Hurley Building & & Rua do Mat\~ao, 1010 \\
Notre Dame, IN, 46556-4618, USA & & S\~ao Paulo, SP, 05508-090, Brazil\\
\emph{E-mail address}: {\tt rbettiol@nd.edu} & & \emph{E-mail address}: {\tt piccione@ime.usp.br}\\
 && \emph{E-mail address}: {\tt sicilian@ime.usp.br}
\end{tabular}
}
\title[Deforming variational problems with varying symmetries]{Deforming solutions of geometric variational problems with varying symmetry groups}
\date{\today}

\thanks{The first named author is supported by the NSF grants DMS-0941615 and DMS-1209387, USA. The second and third named authors are supported by Fapesp and CNPq, Brazil.}

\newcommand{\dd}{\mathrm d}
\newcommand{\vol}{\operatorname{vol}}
\newcommand{\lagsigma}{\mathcal L(\Sigma,M)}
\newcommand{\R}{\mathds R}
\newcommand{\Ric}{\operatorname{Ric}}
\newcommand{\area}{\operatorname{Area}}
\newcommand{\Vol}{\operatorname{Vol}}
\newcommand{\Iso}{\mathrm{Iso}}

\theoremstyle{plain}\newtheorem*{teo_intro}{Theorem}
\theoremstyle{definition}\newtheorem*{defin*}{Definition}

\theoremstyle{plain}\newtheorem{teo}{Theorem}[section]
\theoremstyle{plain}\newtheorem{prop}[teo]{Proposition}
\theoremstyle{plain}\newtheorem{lem}[teo]{Lemma}
\theoremstyle{plain}
\theoremstyle{definition}\newtheorem{defin}[teo]{Definition}
\theoremstyle{remark}\newtheorem{rem}[teo]{Remark}
\theoremstyle{plain} 
\theoremstyle{definition}\newtheorem{example}[teo]{Example}

\theoremstyle{plain} 
\theoremstyle{definition}
\theoremstyle{remark}
\theoremstyle{plain}\newtheorem*{propA}{Proposition}
\theoremstyle{remark}\newtheorem*{exampleA}{Example}

\allowdisplaybreaks
\numberwithin{equation}{section}
\numberwithin{teo}{section}

\begin{document}
\begin{abstract}
We prove an equivariant implicit function theorem for variational problems that are invariant under a varying symmetry group (corresponding to a bundle of Lie groups). Motivated by applications to families of geometric variational problems lacking regularity, several non-smooth extensions of the result are discussed. Among such applications is the submanifold problem of deforming the ambient metric preserving a given variational property of a prescribed family of submanifolds, e.g., constant mean curvature, up to the action of the corresponding ambient isometry groups.
\end{abstract}

\maketitle

\section{Introduction}
Geometric variational problems are typically invariant under a Lie group of symmetries. As a potpourri of examples, closed geodesics $\gamma\colon S^1\to M$ in a Riemannian (or semi-Riemannian) manifold are invariant under rotation of the parameter; constant mean curvature (CMC) hypersurfaces $x\colon N\to M$ are invariant under isometries (i.e., rigid motions) of the ambient space $M$; and harmonic maps $\phi\colon N\to M$ between Riemannian manifolds are invariant under compositions with isometries from both domain and target spaces. These symmetries create an inherent ambiguity among solutions that, a priori, are different from the analytical viewpoint, but rather indistinguishable from the geometric viewpoint. Thus, in order to study deformations of such geometric objects using analytical tools, it is essential to take into account the effect of symmetries. Deformation issues for families of geometric variational problems whose group of symmetries remains \emph{unchanged} were recently studied in \cite{BetPicSic1}, extending classic results of Dancer~\cite{Dan1,Dan2}. The goal of the present paper is to address this issue for a much richer class of deformations, for which the group of symmetries \emph{varies}, in an appropriate sense. Applications are discussed for all the above mentioned examples.

The starting point in studying such deformations is to properly formalize the concept of a smooth (or even continuous) family of symmetry groups $\mathcal G$ parametrized by $\lambda\in\Lambda$. This is achieved considering \emph{bundles of Lie groups} over $\Lambda$, which are Lie groupoids whose source and target maps coincide, see for instance \cite{dl,mac,mo,momr,momr3}. An essential (and somewhat surprising) feature of these objects is that the type of regularity that holds for bundles of Lie groups does not yield smoothness of the global geometric structure of the group, or even continuity of the topological structure, but rather smoothness of the group operations.

As a motivation and companion example, let  $G_\lambda$ be the identity component of the simply-connected $n$-dimensional space form of constant curvature $\lambda$. It is well-known that $G_\lambda$ is a Lie group of dimension $\tfrac12 n(n+1)$ and is isomorphic to $\mathsf{SO}(n,1)$, $\R^n\rtimes\mathsf{SO}(n)$ or $\mathsf{SO}(n+1)$, according to the cases $\lambda<0$, $\lambda=0$ and $\lambda>0$ respectively. Furthermore, note that varying $\lambda\in\R$, the corresponding groups $G_\lambda$ vary from being noncompact to compact as $\lambda$ becomes positive. Nevertheless, $\mathcal G=\{G_\lambda:\lambda\in\R\}$ can still be proved to be a \emph{bundle of Lie groups} (according to Definition~\ref{def:smoothliegps}). This is essentially due to the fact that the corresponding Lie algebras $\mathfrak g_\lambda$ form a smooth subbundle of $\mathfrak{gl}(n+1)$, see Proposition~\ref{thm:critsmoothfamily} and Example~\ref{ex:spaceforms}. Such $\mathcal G$ is an example of a \emph{non-locally trivial} bundle of Lie groups over $\R$. To our knowledge, this example was first studied in \cite{TasUmeYam,UmaYam}, where isometry groups of other symmetric spaces were also considered.

The main result of \cite{UmaYam} provides deformations of a family of CMC immersions of the
$2$-torus $\mathds T^2$ into $\mathds R^3$ to CMC immersions of $\mathds T^2$ into $S^3$ or $\mathds H^3$.
The starting family of immersions into $\mathds R^3$,  whose existence has been first proved by H. Wente \cite{wente1984}, is of significant importance in Differential Geometry, as it provided a long-searched counter-example to a famous conjecture of H. Hopf. Smoothness of the family of ambient isometry groups when passing from positive to negative curvature has been employed in \cite{UmaYam} to setup the framework for an implicit function theorem for CMC immersions of $\mathds T^2$ into $3$-dimensional space forms, up to rigid motions.
Similar situations in which the ambient space is deformed to yield a corresponding family of CMC (in this case, \emph{minimal}) submanifolds are found in \cite{hw2,mora}. These contributions provide, respectively, a deformation of minimal helicoids with handles from $S^2(r)\times\mathds R$ to $\mathds R^3$ as $r\to+\infty$, and of compact pieces of the Costa-Hoffman-Meeks surface in $\mathds R^3$ to $\mathds H^2\times\mathds R$. The purpose of this paper is to develop this type of deformation technique in a general variational context, dealing with perturbations of geometric variational problems that entail a change of the symmetry group.

Equipped with the above notion of varying symmetry groups, we now explain our main result. Consider a family of smooth functionals parametrized by $\lambda\in\Lambda$,
\begin{equation*}
f_\lambda\colon X\to\R, \quad f_\lambda(g_\lambda\cdot x)=f_\lambda(x)\quad  \mbox{ for all } x\in X, \, g_\lambda\in G_\lambda,
\end{equation*}
where $X$ is a Banach manifold on which the bundle of Lie groups $\mathcal G=\{G_\lambda:\lambda\in\Lambda\}$ acts by diffeomorphisms, i.e., $X$ has a $G_\lambda$-action for all $\lambda\in\Lambda$. We are interested in solutions of the equation
\begin{equation}\label{eq:critpts}
\dd f_\lambda(x)=0.
\end{equation}
Suppose the $\mathcal G$-action on $X$ is \emph{regular}, in the sense that the map
\begin{equation*}
\beta_x^\lambda\colon G_\lambda\to X, \quad \beta_x^\lambda(g)=g\cdot x,
\end{equation*}
is differentiable at the identity $1_\lambda\in G_\lambda$, for all $\lambda\in\Lambda$, and the section $(\lambda,x)\mapsto\dd\beta_x^\lambda(1_\lambda)$ is continuous.\footnote{$\dd\beta_x^\lambda(1_\lambda)$ is a section of the vector bundle $\mathrm{Hom}(\mathfrak g,TX)$ over $\Lambda\times X$, whose fiber over $(\lambda,x)$ is the space of linear maps from $\mathfrak g_\lambda$ to $T_xX$.} Denote by $\mathcal D_x^\lambda\subset T_xX$ the image of $\dd\beta_x^\lambda(1_\lambda)$, which coincides with the tangent space at $x$ to its $G_\lambda$-orbit. Note that since $f_\lambda$ is $G_\lambda$-invariant, its linearization $\dd f_\lambda\colon X\to TX^*$ is $G_\lambda$-equivariant.\footnote{Denoting by $\gamma^\lambda_g\colon X\to X$ the diffeomorphism $X\ni x\mapsto g \cdot x\in X$, where $g\in G_\lambda$, there is a action of $G_\lambda$ on $TX$, given by $g\cdot v=\dd\gamma_g^\lambda(x)v\in T_{gx}X$ for all $v\in T_xX$ and $g\in G_\lambda$; and an induced action on the cotangent bundle $TX^*$, given by $g\cdot\alpha=\alpha(g^{-1}\,\cdot)\in T_{gx}X^*$, for all $\alpha\in T_xX^*$.}
In particular, if $x\in X$ is a critical point of $f_\lambda$, then the entire $G_\lambda$-orbit of $x$ consists of critical points of $f_\lambda$. Moreover, the kernel of the second derivative $\dd^2\!f_\lambda(x)$ contains $\mathcal D_{x}^\lambda$. The critical point $x$ is called \emph{equivariantly nondegenerate} if the reverse
inclusion also holds, i.e., the kernel of $\dd^2\!f_\lambda(x)$ is the smallest possible. In this situation, our main abstract result is the following equivariant implicit function theorem:

\begin{teo_intro}
Let $(\lambda_0,x_0)\in\Lambda\times X$ be a solution of \eqref{eq:critpts}, where $x_0$ is an equivariantly nondegenerate critical point of $f_{\lambda_0}$. Assume that the second variation $\dd^2 \! f_{\lambda_0}(x_0)$ is represented by a self-adjoint Fredholm operator.\footnote{For details on these technical regularity assumptions, see Section~\ref{sec:proofregular}.} Then, there exist an open neighborhood $U\subset\Lambda$ of $\lambda_0$, an open neighborhood $V\subset X$ of $x_0$ and a map $U\ni\lambda\mapsto x_\lambda\in V$ such that $(\lambda,x)\in U\times V$ is a solution of \eqref{eq:critpts} if and only if $x$ is in the $G_\lambda$-orbit of $x_\lambda$.
\end{teo_intro}

A proof of the above result will be given in steps, according to the level of regularity assumptions made
on the group action. We stress that in many variational problems, especially those involving
actions of diffeomorphism groups (e.g., problems invariant under reparametrizations), the canonical
symmetry group \emph{does not act smoothly}. This subtle issue originates from the loss of differentiability caused by the chain rule, in an action by composition. A paradigmatic example of this situation is given by the CMC embedding problem, where the ambient isometry group acts on unparameterized embeddings, see Subsection~\ref{subsec:cmc} and \cite{BetPicSic1}.

A first proof of the above result, for the regular case, is given in Section~\ref{sec:proofregular}.
Here, the main issue concerns the appropriate notion of Fredholmness that must be used. Restricting to the realm of \emph{Hilbert} manifolds, where a standard Fredholmness assumption for $\dd^2\! f_{\lambda_0}(x_0)$ would be feasible, is not suited for most geometric variational problems that are naturally
defined on \emph{Banach} manifolds. In this situation, there is no clear Fredholmness condition that can be required for $\dd^2\! f_{\lambda_0}(x_0)$, given that in general there are no Fredholm operators from a Banach space to its dual.
In order to deal with this problem, one often uses an auxiliary pre-Hilbert structure underlying the functional space
(typically, an $L^2$-pairing), relative to which one requires Fredholmness (see Definition~\ref{def:auxFredholm}). From a functional analytical viewpoint, an interesting observation is that Fredholmness and symmetry with respect to an incomplete inner product in general  \emph{does not} imply the vanishing of the Fredholm index, as it happens for self-adjoint operators
on Hilbert spaces. This is discussed in Appendix~\ref{app}, where a criterion for the vanishing of Fredholm index of
symmetric operators is given in terms of a certain \emph{ellipticity} condition.

The essential ingredient for the proof in the regular case is the existence of a submanifold of $X$ which is
\emph{transversal} to the group orbit $G_{\lambda_0}(x_0)$ at $x_0$. As to the nonregular case, discussed in Section~\ref{sec:nonregular}, the transversality argument is replaced by a topological degree argument,
whose abstract formulation is given in Proposition~\ref{thm:nonemptyintersection_fiber}.
For such general result, it is only required that the group action is differentiable in a dense subset of $X$, which happens to be the case in a large class of
geometric variational problems. Roughly speaking, when the functional space consists of $C^k$-submanifolds $\Sigma$ of a smooth manifold $M$, and the action is given by applying diffeomorphisms of $M$, differentiability occurs at those $\Sigma$ that are $C^{k+1}$.

Several applications of the above abstract result are discussed in Section~\ref{sec:applications}, in the context of geometric variational problems. First, we deal with the above mentioned perturbation
problem of CMC hypersurfaces in families of Riemannian manifolds with smoothly varying isometry groups.
We also consider the case of \emph{immersed} submanifolds, possibly with boundary.
Second, we discuss perturbations of harmonic maps between Riemannian manifolds, where the metrics vary on both target and source manifolds. Finally, we consider a perturbation
problem for Hamiltonian stationary Lagrangian submanifolds of a symplectic manifold $(M,\omega)$.
Here, we consider a family of Riemannian metrics on $M$ whose isometry groups act in a Hamiltonian fashion on $(M,\omega)$, obtaining a perturbation result of Lagrangian submanifolds that are Hamiltonian stationary with respect to a K\"ahler metric. This improves a result recently obtained in \cite{BetPicSan12}.

Finally, a brief comment on the smoothness of families of Lie groups is in order. We point out that, despite the
fact that the general Lie algebroid and Lie groupoid theories provide the existence of smooth families
of Lie groups associated to bundles of Lie algebras, this abstract result is not suited for
the purposes of the present paper. Namely, we observe that this only provides bundles of simply-connected
Lie groups, unrelated to each other, except for the smoothness condition. Moreover, the result cannot
 be employed in order to establish whether a given family of (possibly not simply-connected) Lie groups
is smooth, which is the central point that must be addressed here.
This suggests the formulation of a smoothness criterion for families of Lie subgroups of a fixed Lie group,
that we discuss in Proposition~\ref{thm:critsmoothfamily}.
An application of this criterion (Example~\ref{ex:spaceforms}) provides an alternative and more conceptual proof of \cite[Thm~3.1]{UmaYam}.

\begin{section}{Bundles of Lie groups}
\label{sec:bundlesliegps}

As mentioned in the Introduction, we are interested in studying problems invariant under a \emph{family} of Lie groups, rather than a fixed Lie group. Such concept appears in the literature with slight variations, see \cite{coppersmith,dl,mac,mo,richardson,weinstein}.
In order to give a precise definition, recall that a \emph{Lie groupoid}\footnote{An abstract \emph{groupoid} is a category in which every arrow is invertible.} $\mathcal G$ over a (connected) smooth manifold $\Lambda$ is a smooth manifold equipped with the following data:
\begin{itemize}
\item two smooth submersions $\mathfrak s,\mathfrak t\colon \mathcal G\to \Lambda$, called \emph{source} and \emph{target} maps;\footnote{A point $g\in\mathcal G$ with $\mathfrak s(g)=x$ and $\mathfrak t(g)=y$ should be, categorically, thought of as an \emph{arrow} from $x$ to $y$.}
\item an associative \emph{composition} operation $\mathcal G\times_\Lambda\mathcal G\ni (g,h)\to g\cdot h \in\mathcal G$, which is a smooth map that associates to each pair $(g,h)$ with $\mathfrak s(g)=\mathfrak t(h)$ the composition $g\cdot h$;
\item an involution $i\colon \mathcal G\to\mathcal G$, which is a smooth map that associates to each $g\in\mathcal G$ its composition inverse $i(g)=g^{-1}$;
\item a smooth map $s_\mathrm I\colon \Lambda\to\mathcal G$, that to each $\lambda\in\Lambda$ associates to the unit $s_\mathrm I(\lambda)=1_\lambda$ for composition.
\end{itemize}

\begin{defin}\label{def:smoothliegps}
A Lie groupoid $\mathcal G$ over $\Lambda$ for which the source and target maps coincide, i.e., $\mathfrak s=\mathfrak t$, is called a \emph{(smooth) bundle of Lie groups over $\Lambda$}, or a \emph{(smooth) family of Lie groups parameterized by $\Lambda$}. In this case, we also write $\mathcal G=\{G_\lambda:\lambda\in\Lambda\}$, where $G_\lambda:=\mathfrak s^{-1}(\lambda)$ the Lie group consisting of the inverse image of $\lambda$ by $\mathfrak s$.
\end{defin}

As a first example, consider the trivial of bundle of Lie groups $\mathcal G=\Lambda\times G$, where $G$ is a Lie group and $\mathfrak s=\mathfrak t$ is the projection onto the first factor. In this case, all groups $G_\lambda$ are isomorphic to $G$. Due to a result of Weinstein~\cite{weinstein}, see also Moerdijk \cite[Sec 1.3 (c)]{mo}, the same happens for possibly nontrivial bundles of Lie groups, provided all $G_\lambda$ are compact.

\begin{prop}
A bundle of Lie groups $\mathcal G$ over $\Lambda$ such that $\mathfrak s\colon\mathcal G\to\Lambda$ is a proper map must be \emph{locally trivial}, i.e., every $\lambda\in\Lambda$ has an open neighborhood $U$ in $\Lambda$, such that $\mathfrak s^{-1}(U)=U\times G_\lambda$. In particular, bundles of \emph{compact} Lie groups are locally trivial.
\end{prop}

Consequently, since $\Lambda$ is assumed connected, the groups $G_\lambda$ of a bundle of compact Lie groups are pairwise isomorphic.
From the viewpoint of varying families of symmetry groups that we adopt in our applications, this still is a trivial situation. Namely, the symmetry group remains unchanged (up to isomorphism). Therefore, actual nontrivial cases can only occur when some of the groups $G_\lambda$ are noncompact. In this case, the different groups $G_\lambda$ need not be isomorphic, not even homotopy equivalent (see Example~\ref{ex:spaceforms}).

\begin{rem}
Similar definitions for bundles of Lie groups can be given in other regularity contexts. For example, this concept was considered in the analytic category in \cite{richardson}. The main result of \cite{richardson} was later proved also in the smooth category~\cite{coppersmith}. For our applications, it will be useful to also have a notion of \emph{continuous} bundles of Lie groups (as opposed to the above \emph{smooth} notion), see Definition~\ref{def:continuousliegps}.
\end{rem}

\subsection{Integrating bundles of Lie algebras}
In order to produce nontrivial examples of bundles of Lie groups, we consider the problem of obtaining such bundles from bundles of Lie algebras, in a fashion similar to Lie's Third Theorem (but for a $1$-parameter family of Lie algebras).  Given a bundle of Lie groups $\mathcal G$, denote by $\mathrm{Ver}(\mathcal G)$ the \emph{vertical bundle of $\mathcal G$}, i.e., the vector bundle over $\mathcal G$ whose fiber over $g\in\mathcal G$ is $\ker\dd\pi(g)\subset T_g\mathcal G$, which is the tangent space at $g$ to $G_{\mathfrak s(g)}=\mathfrak s^{-1}(\mathfrak s(g))$. The pull-back of this vector bundle by the identity section,
\begin{equation}\label{eq:bdlliealg}
\mathrm{Lie}(\mathcal G):=s_\mathrm I^*\big(\mathrm{Ver}(\mathcal G)\big),
\end{equation}
is a vector bundle over $\Lambda$, called the \emph{bundle of Lie algebras associated to the bundle of Lie groups $\mathcal G$}. The fiber $\mathrm{Lie}_\lambda(\mathcal G)$ of \eqref{eq:bdlliealg} over $\lambda\in\Lambda$ is the Lie algebra of the Lie group $G_\lambda$. In the general theory of Lie groupoids, this construction is a special case of the \emph{Lie algebroid associated to a Lie groupoid}, see \cite[Sec 6.1]{momr}.

An abstract Lie algebroid (see \cite[Sec 6.2]{momr} for definitions) is \emph{integrable} if it is isomorphic to the Lie algebroid associated to a Lie groupoid. There exist (finite-dimensional) Lie algebroids that are not integrable \cite[Sec 6.4]{momr}, and finding general integrability criterions for Lie algebroids is a notoriously difficult problem, see \cite{cf,momr2,momr3}. Nevertheless, in the simpler context of bundles of Lie algebras, Douady and Lazard~\cite{dl} proved that \emph{any} (finite-dimensional) bundle of Lie algebras can be integrated to a bundle of Lie groups (which may be not locally trivial nor Hausdorff), cf. \cite[Ex 6.3 (2)]{momr}.
We now present a simple proof of a particular case of this result (in which the bundle obtained is Haussdorff, but possibly not locally trivial). This particular case comes from integrating bundles of Lie \emph{subalgebras} of a given Lie algebra $\widehat{\mathfrak g}$, cf. \cite{momr3}.

\begin{prop}\label{thm:critsmoothfamily}
Let $\Lambda$ be a manifold and $\widehat G$ be a Lie group with Lie algebra $\widehat{\mathfrak g}$. Let $\mathfrak g$ be a smooth subbundle of the trivial vector bundle $\Lambda\times\widehat{\mathfrak g}$ over $\Lambda$, such that, for all $\lambda\in \Lambda$, the fiber $\mathfrak g_\lambda$ is a Lie subalgebra of $\widehat{\mathfrak g}$. For all $\lambda\in \Lambda$, let $G_\lambda$ be the connected subgroup of $\widehat G$ whose tangent space at $1$ is $\mathfrak g_\lambda$, and set
\begin{equation*}
\mathcal G:=\bigcup\limits_{\lambda\in \Lambda}\left(\{\lambda\}\times G_\lambda\right)\subset \Lambda\times \widehat G.
\end{equation*}
Assume that for all $\lambda_0\in \Lambda$, there exists a submanifold $A\subset\widehat G$ and an open neighborhood $V$ of $\lambda_0$ in $\Lambda$ such that:
\begin{itemize}
\item[(a)] $1\in A$;
\item[(b)] $T_1A\oplus\mathfrak g_{\lambda_0}=\widehat{\mathfrak g}$;
\item[(c)] for all $\lambda\in V$, $G_\lambda\cap A=\{1\}$.
\end{itemize}
Then, $\mathcal G$ is a bundle of Lie groups over $\Lambda$, whose associated bundle of Lie algebras is $\mathrm{Lie}(\mathcal G)=\mathfrak g$.
\end{prop}

\begin{proof}
It suffices to prove that $\mathcal G$ is a submanifold of $\Lambda\times\widehat G$. Define a smooth map
\begin{equation*}
f\colon\mathfrak g\times A\to \Lambda\times\widehat G, \quad f(\lambda,X,a)=\big(\lambda,a\cdot\exp(X)\big).
\end{equation*}
For all $\lambda_0\in \Lambda$, using (b),  it is easily seen that $\dd f(\lambda_0,0,1)\colon T_{(\lambda_0,0)}\mathfrak g\oplus T_1A\to T_{\lambda_0}\Lambda\oplus\widehat{\mathfrak g}$ is
an isomorphism.\footnote{%
The tangent space at the point $(\lambda_0,0)$ of $\mathfrak g$ is canonically identified with the horizontal plus vertical direct sum $T_{\lambda_0}\Lambda\oplus\mathfrak g_{\lambda_0}$. For $v\in T_{\lambda_0}\Lambda$, $Z\in\mathfrak g_{\lambda_0}$ and $w\in T_1A$, $\mathrm df(\lambda_0,0,1)(v,Z,w)=(v,Z+w)$, and this map is an isomorphism because of (b).}
Thus, by the inverse function theorem (up to reducing the size of $A$), $f$ restricts to a diffeomorphism from
the product $W\times A$ to an open subset $B=f(W\times A)$ of $\Lambda\times\widehat G$, where $W$ is a neighborhood of $(\lambda_0,0)$ in $\mathfrak g$. We claim that the preimage $f^{-1}(\mathcal H\cap B)$ is given by $W\times\{1\}$.
Namely, $\big(\lambda,a\cdot\exp(X)\big)\in\mathcal G$ if and only if $a\cdot\exp(X)\in\mathcal G_\lambda$, and since $\exp(X)\in\mathcal G$,
this occurs if and only if $a\in G_\lambda$. By assumption (c), this implies $a=1$, proving our claim.
It follows that $\mathcal G$ is a smooth submanifold of $\Lambda\times\widehat G$ near all points of the form $(\lambda_0,1)$.

In order to determine the local structure of $\mathcal G$ at the other points, let us use the following fact, proved in Lemma~\ref{thm:existencesections} below. For all $\lambda\in \Lambda$ and all $h\in G_\lambda$, there exists a smooth function $s\colon U\to\widehat G$ defined in an open neighborhood $U$ of $\lambda$,
such that $s(\lambda)=h$ and $s(y)\in G_y$ for all $y\in U$. Fixed an arbitrary $\lambda$ and an arbitrary $h\in G_\lambda$, let $s\colon U\to\widehat G$ be any such function, and consider the diffeomorphism\footnote{$\phi_s$ is clearly smooth, and its inverse is $\phi_s^{-1}(y,g)=\big(y,s(y)^{-1}g\big)$, which is also smooth.}
\begin{equation*}
\phi_s\colon U\times\widehat G\to U\times\widehat G, \quad \phi_s(y,g)=\big(y,s(y)\cdot g\big).
\end{equation*}
Such a map carries a neighborhood of $(\lambda,1)$ to a neighborhood of $(\lambda,h)$. Moreover, it preserves $\mathcal G\cap(U\times\widehat G)$, because $s(y)\in G_y$ for all $y\in U$. Since sufficiently small neighborhoods in $\mathcal G$ of all points of the form $(\lambda,1)$ are submanifolds of $\Lambda\times\widehat G$, and $h$ is arbitrary, it follows that $\mathcal G$ is a submanifold of $\Lambda\times\widehat G$.
\end{proof}

We have used the following existence result of smooth local sections of $\mathcal G$, with prescribed value at some given point:
\begin{lem}\label{thm:existencesections}
In the notations above, for all $\lambda\in \Lambda$ and $h\in G_\lambda$, there exists an open neighborhood $U$ of $\lambda$ in $\Lambda$ and a smooth function $s\colon U\to\widehat G$ such that $s(\lambda)=h$ and $s(y)\in G_y$ for all $y\in U$.
\end{lem}
\begin{proof}
Since $G_\lambda$ is connected, there exist a finite sequence $X_1,\ldots X_n$ in $\mathfrak g_\lambda$ such that $h=\exp(X_1)\cdot\ldots\cdot\exp(X_n)$, since every neighborhood of $1$ generates $G_\lambda$. Then, one can find smooth extensions $u_1,\ldots,u_n$ of the $X_i$'s to smooth local sections of the vector bundle $\mathfrak g$. The map $s(y)=\exp\big(u_1(y)\big)\cdot\ldots\cdot\exp\big(u_n(y)\big)$ is the desired smooth local section of $\mathcal G$.
\end{proof}

We now discuss an example to which the above result applies, concerning the bundle of Lie groups formed by the (identity connected component of) isometry groups of simply-connected space forms of constant curvature $\lambda\in\R$.
\begin{example}\label{ex:spaceforms}
Given $n\ge2$, set $\widehat{\mathfrak g}=\mathfrak{gl}(n+1)$ and $\widehat G=\mathsf{GL}(n+1)$. For each $\lambda\in\R$, consider the injective linear map
\begin{equation}
L_\lambda\colon \mathfrak{so}(n)\times\R^n\to\widehat{\mathfrak g}, \quad L_\lambda(D,u)=\begin{pmatrix}0&-\lambda u^\mathrm t\\u&D\end{pmatrix},
\end{equation}
where $D\in\mathfrak{so}(n)$ (i.e., $D$ is an anti-symmetric matrix), $u$ is a column vector and
$u^\mathrm t$ is the transpose of $u$. The map $\R\ni \lambda\mapsto L_\lambda\in\mathrm{Hom}\big(\mathfrak{so}(n)\times\mathds R^n,\widehat{\mathfrak g}\big)$ is smooth, thus the images of $L_\lambda$ form a smooth family of subspaces of $\widehat{\mathfrak g}$.
For all $\lambda\in\mathds R$, the image of $L_\lambda$ is a Lie subalgebra $\mathfrak g_\lambda$ of $\widehat{\mathfrak g}$ with dimension $\frac12n(n+1)$, and $\mathfrak g=\bigcup_{\lambda\in\R}(\{\lambda\}\times\mathfrak g_\lambda)$ is a smooth vector subbundle of the trivial bundle $\R\times\widehat{\mathfrak g}$.

For $\lambda\neq0$, it is easy to see that $\mathfrak g_\lambda$ is the Lie algebra of the Lie group $\mathsf{SO}(\eta_\lambda)\subset\widehat G$ of
isomorphisms of $\R^{n+1}$ that preserve the nondegenerate symmetric bilinear form $\eta_\lambda$, whose matrix in the canonical basis of $\R^{n+1}$ is:
\begin{equation*}
\eta_\lambda\cong \begin{pmatrix}\frac1{\sqrt \lambda}&0\\0&\sqrt \lambda\,\mathrm I_n\end{pmatrix}  \mbox{ for } \lambda>0, \quad \mbox{ and }\quad\eta_\lambda\cong\begin{pmatrix}\frac{-1}{\sqrt{-\lambda}}&0\\0&\sqrt{-\lambda}\,\mathrm I_n\end{pmatrix} \mbox{ for } \lambda<0,
\end{equation*}
where $\mathrm I_k$ denotes the $k\times k$ identity matrix. Set $G_\lambda:=\mathsf{SO}(\eta_\lambda)$ if $\lambda\neq0$, and $G_0:=\begin{pmatrix}1&0\\\mathds R^n&\mathsf{SO}(n)\end{pmatrix}\subset
\widehat G$. Then, we have the following isomorphisms with the isometry groups of space forms:
\begin{equation*}
G_\lambda\cong\begin{cases} \mathsf{SO}(n+1), & \mbox{(identity component of) isometry group of } S^n, \mbox{  if } \lambda>0,\\
\R^n\rtimes\mathsf{SO}(n),  &\mbox{(identity component of) isometry group of } \R^n, \mbox{  if } \lambda=0,\\
\mathsf{SO}(n,1),  &\mbox{(identity component of) isometry group of } \mathds H^n, \mbox{  if } \lambda<0.\\
\end{cases}
\end{equation*}

We claim that
\begin{equation}\label{eq:spaceformisomgps}
\mathcal G:=\bigcup_{\lambda\in\R}\left(\{\lambda\}\times G_\lambda\right)\subset\R\times\widehat G
\end{equation}
is a bundle of Lie groups over $\R$. For all $\lambda\in\mathds R$, the subspace $\mathfrak a\subset\widehat{\mathfrak g}$ given by $\mathfrak a=\begin{pmatrix}\R&\R^n\\0&\mathrm{Sym}_n\end{pmatrix}$ is a complement of $\mathfrak g_\lambda$ in $\widehat{\mathfrak g}$, where $\mathrm{Sym}_n$ denotes the space of $n\times n$ symmetric matrices. Denote by $A$ the submanifold of $\widehat G$ consisting of matrices of the form $\begin{pmatrix} a&v^\mathrm t\\0&B\end{pmatrix}$, where $a$ is a positive real number, $v$ is an arbitrary vector in $\mathds R^n$ and $B$ is a positive symmetric matrix. Clearly, $1\in A$ and $T_1A=\mathfrak a$. It is a straightforward computation\footnote{%
For $\lambda\ne0$, $G_\lambda$ consists of matrices of the form $\begin{pmatrix}a&v^\mathrm t\\w&B\end{pmatrix}$, where $a\in\R$, $v,w\in\R^n$, $B\in\mathfrak{gl}(n)$, $av+\lambda B^\mathrm tw=0$, $a^2+\lambda\Vert w\Vert^2=1$, $vv^\mathrm t+\lambda B^\mathrm tB=\lambda\, I_n$. Given $\begin{pmatrix}a&v^\mathrm t\\w&B\end{pmatrix}\in G_\lambda\cap A$,
one sees easily that it must be $w=0$, $v=0$, $a=1$ and $B^2=\mathrm I_n$. Since $B$ is (symmetric and) positive, the identity $B^2=\mathrm I_n$ implies $B=\mathrm I_n$.
Thus, $A\cap G_\lambda=\mathrm I_{n+1}$ for $\lambda\ne0$. Similarly, $A\cap G_0=\mathrm I_{n+1}$.} that $A\cap G_\lambda=\{1\}$ for all $\lambda\in\R$. Thus, the above claim follows from Proposition~\ref{thm:critsmoothfamily}.
\end{example}

\begin{rem}
Note that the above bundle of Lie groups \eqref{eq:spaceformisomgps} is not locally trivial at $0\in\R$. This is clear since $G_\lambda$ is noncompact for $\lambda\leq0$ and compact for $\lambda>0$. Moreover, $G_{\lambda_1}$ and $G_{\lambda_2}$ are not even homotopy equivalent when $\lambda_1\leq0<\lambda_2$.
\end{rem}

\begin{rem}
A totally analogous construction as in Example~\ref{ex:spaceforms} can be done for the isometry group
of symmetric spaces, as in \cite{TasUmeYam}. Let $(M,\mathbf g)$ be a compact symmetric space, and choose $p\in M$. Denote by $\mathfrak g$ the Lie algebra of the isometry group of $(M,\mathbf g)$ and by $\mathfrak k$ the Lie algebra of
the stabilizer of $p$. Consider the reductive decomposition $\mathfrak g=\mathfrak k\oplus\mathfrak m$, with $[\mathfrak k,\mathfrak k]\subset\mathfrak k$,
$[\mathfrak k,\mathfrak m]\subset\mathfrak m$ and $[\mathfrak m,\mathfrak m]\subset\mathfrak k$. Then, for $\lambda\in\mathds R$, one
considers a \emph{deformation} $[\cdot,\cdot]_\lambda$ of the Lie bracket $[\cdot,\cdot]$ of $\mathfrak g$, by setting:
\[[x,y]_\lambda=[x,y],\quad[x,v]_\lambda=[x,v],\quad[v,w]_\lambda=\lambda\cdot[v,w],\, \mbox{ for all }\,x,y\in\mathfrak k,\ v,w\in\mathfrak m.\]
Clearly, $[\cdot,\cdot]_\lambda$ is equal to $[\cdot,\cdot]$ when $\lambda=1$, while for $\lambda=-1$ it is the Lie bracket on the Lie algebra of the isometry group of the noncompact dual of $(M,\mathbf g)$. Denote by $\mathfrak g_\lambda$ the Lie algebra $\mathfrak g$ endowed with the
Lie bracket $[\cdot,\cdot]_\lambda$. Then, as in Example~\ref{ex:spaceforms}, $\mathfrak g_\lambda$ is isomorphic to $\mathfrak g_1$, $\mathfrak m\rtimes\mathfrak k$ or $\mathfrak g_{-1}$, according to the cases $\lambda>0$, $\lambda=0$ and $\lambda<0$ respectively.
\end{rem}
\end{section}

\begin{section}{The regular case}
\label{sec:proofregular}
Let $X$ be a Banach manifold, $f\colon X\to\R$ a sufficiently regular function and $x_0\in X$ a critical point of $f$.
It is generally not a convenient assumption that the second derivative $\dd^2\! f(x_0)\colon T_{x_0}X\to T_{x_0}X^*$
is a Fredholm map. Namely, most Banach spaces do not admit \emph{any} Fredholm operator to their dual.
On the other hand, in the cases of interest for applications to geometric variational problems, second derivatives are represented by differential operators that \emph{are} Fredholm when considered \emph{between suitable spaces}.
For instance, strongly elliptic self-adjoint differential operators of order $d$, defined on the space
of sections of some vector bundle $E$, are Fredholm operators of index $0$ from $\Gamma^{k,\alpha}(E)$
to $\Gamma^{k-d,\alpha}(E)$, for every $d\leq k$ and every $\alpha\in\left]0,1\right[$, where $\Gamma^{k,\alpha}(E)$ denotes the space of sections of $E$ of class $C^{k,\alpha}$, see for instance \cite{Whi2}.
A natural abstract formulation of such Fredholmness property is given as follows.

\begin{defin}\label{def:auxFredholm}
Given a Banach manifold $X$ and a $C^{k+1}$-function $f\colon X\to\R$, $k\ge1$,
an \emph{auxiliary Fredholm structure} for $(X,f)$ consists of:
\begin{itemize}
\item a smooth Banach bundle $\mathcal E$ on $X$, with continuous inclusion $TX\subset\mathcal E$;
\item a (possibly non-complete) inner product $\langle\cdot,\cdot\rangle_x$ on each fiber $\mathcal E_x$, for all $x\in X$,
\end{itemize}
satisfying the following properties:
\begin{itemize}
\item there exists a $C^k$-section $\delta f\colon X\to\mathcal E$ such that $\mathrm df(x)=\langle\delta f(x),\cdot\rangle_x$
for all $x\in X$ (in particular, $\mathrm df(x)=0$ if and only if $\delta f(x)=0$);
\item for all critical points $x_0\in X$, there exists a Fredholm operator $J_{x_0}\colon T_{x_0}X\to\mathcal E_{x_0}$
of index $0$, such that
\begin{equation}\label{eq:abstractJacobi}
\mathrm{Hess}^f(x_0)[v,w]=\langle J_{x_0}(v),w\rangle_{x_0},\quad\mbox{ for all } v,w\in T_{x_0}X,
\end{equation}
where $\mathrm{Hess}^f(x_0)$ is the bounded symmetric
bilinear form on $T_{x_0}X$ given by the Hessian of $f$ at $x_0$.
\end{itemize}
\end{defin}

Note that $\ker \big(\mathrm{Hess}^f(x_0)\big)=\ker(J_{x_0})$, and that the map $\delta f$ is a \emph{gradient-like} map for $f$.
From \eqref{eq:abstractJacobi}, $J_{x_0}$ is symmetric with respect to $\langle\cdot,\cdot\rangle_{x_0}$. Nevertheless, this symmetry alone does not imply in general that $J_{x_0}$ has Fredholm index $0$, see Appendix~\ref{app} for a thorough discussion. In typical applications, $\delta f$ is a quasi-linear elliptic differential operator between suitable spaces and $J_{x_0}$ is its linearization, i.e., a linear elliptic
differential operator, for which the above assumptions are satisfied.
\begin{rem}
For all $x_0\in X$, denote by $0_{x_0}$ the zero of the fiber $\mathcal E_{x_0}$, and by
$\pi_{\mathrm{ver}}(x_0)\colon T_{0_{x_0}}\mathcal E\to\mathcal E_{x_0}$ the vertical projection.
Given a critical point $x_0$ of $f$, i.e., $\delta f(x_0)=0_{x_0}$, the derivative $\mathrm d(\delta f)(x_0)$ is a linear map from $T_{x_0}X$ to $T_{0_{x_0}}\mathcal E$. The operator $J_{x_0}$ above can be defined as the \emph{vertical derivative} of
$\delta f$ at $x_0$, given by:
\begin{equation}\label{eq:ddfjacobi}
\pi_{\mathrm{ver}}(x_0)\circ\big(\mathrm d(\delta f)(x_0)\big)\colon T_{x_0}X\longrightarrow\mathcal E_{x_0}.
\end{equation}
\end{rem}

Finally, let us give a fiber bundle version of a result in \cite{BetPicSic1} on the intersection of Banach submanifolds. Aiming at applications to not necessarily smooth group actions, we only assume differentiability along one fiber.

\begin{prop}\label{thm:nonemptyintersection_fiber}
Let $\mathcal M$ be the total space of a fiber bundle with base $A$. For all $a\in A$, denote by $M_a$
the fiber over $a$, which is assumed to be a finite-dimensional manifold.
Let $N$ be a (possibly infinite-dimensional) Banach manifold,
and $Q\subset N$ a Banach submanifold. Assume that $\chi\colon\mathcal M\to N$ is a continuous
function such that there exists $a_0\in A$ and $p_0\in M_{a_0}$ with:
\begin{itemize}
\item[(a)] $\chi(p_0)\in Q$;
\item[(b)] $\chi_{a_0}=\chi\big\vert_{M_{a_0}}\!\!\colon M_{a_0}\to N$ is of class $C^1$;
\item[(c)] $\mathrm d\chi_{a_0}\big(T_{p_0}M_{a_0}\big)+T_{\chi(p_0)}Q=T_{\chi(p_0)}N$.
\end{itemize}
Then, for $a\in A$ near $a_0$, $\chi(M_a)\cap Q\ne\emptyset$.
\end{prop}
\begin{proof}
Given the local character of the result, we can use a trivialization of the fiber bundle $\mathcal M$ around
$a_0$, and assume that $\mathcal M=A\times M$, where $M$ is a finite-dimensional manifold diffeomorphic to the fibers.
For this case, the proof of the result is given in \cite[Prop~3.4]{BetPicSic1}, using a topological degree argument.
Note that transversality arguments \emph{cannot be used} here, since we are not assuming differentiability of $\chi$ along
\emph{all} the fibers in a neighborhood of $a_0$.
\end{proof}

Given Banach spaces $E$ and $F$, denote by $\mathrm{Hom}(E,F)$ the space of bounded linear operators from $E$ to $F$, and for $T_j\in\mathrm{Hom}(E_j,F)$, define $(T_1\oplus T_2)(e_1,e_2):=T_1(e_1)+T_2(e_2)$.

\begin{lem}
\label{thm:opensurjective}
Let $E_1$, $E_2$ and $F$ be Banach spaces. The following is open in $\mathrm{Hom}(E_1,F)\times\mathrm{Hom}(E_2,F)$:
\begin{multline*}\big\{(T_1,T_2)\in\mathrm{Hom}(E_1,F)\times\mathrm{Hom}(E_2,F): (T_1\oplus T_2)\colon E_1\oplus E_2\to F\\ \text{is surjective and has complemented kernel}\big\}.\end{multline*}
\end{lem}
\begin{proof}
Set $E=E_1\oplus E_2$; it is well-known that the set of $T\in\mathrm{Hom}(E,F)$ that are surjective and have complemented kernel is open in $\mathrm{Hom}(E,F)$. Moreover, the map $\mathrm{Hom}(E_1,F)\times\mathrm{Hom}(E_2,F)\ni (T_1,T_2)\mapsto T_1\oplus T_2\in\mathrm{Hom}(E,F)$
is linear and continuous.
\end{proof}

We now give the proof of the Theorem in the Introduction, when all the above regularity hypotheses are fulfilled.

\begin{proof}[Proof (Regular case)]
Recall that we are denoting by $\mathcal D_x^\lambda$ the image of the linear map
$\mathrm d\beta_x^\lambda(1)\colon\mathfrak g_\lambda\to T_xX$. Our assumptions on the regularity of the action
imply that $\mathrm d\beta_x^\lambda(1)$ depends continuously on $(x,\lambda)$.
We claim that we can find a Banach submanifold $S\subset X$ through $x_0$, satisfying the following properties:
\begin{enumerate}
\item\label{itm:S1} $T_{x_0}S\oplus\mathcal D_{x_0}^{\lambda_0}=T_{x_0}X$;
\item\label{itm:S2} $T_xS+\mathcal D_x^\lambda=T_xX$ for all $x\in S$ and all $\lambda$ in a neighborhood $U\subset\Lambda$ of $\lambda_0$;
\item\label{itm:S3} there exists a neighborhood $\widetilde V\subset X$ of $x_0$ such that $G_\lambda\cdot S\supset\widetilde V$ for all $\lambda\in U$.
\end{enumerate}
The existence of such $S$ can be argued as follows. First, observe that $\mathcal D_{x_0}^{\lambda_0}$ has finite dimension, thus it is complemented
in $T_{x_0}X$; let $\mathcal S$ be any closed complement of $\mathcal D_{x_0}^{\lambda_0}$ in $T_{x_0}X$, and let $S\subset X$ be any smooth
submanifold through $x_0$ with $T_{x_0}S=\mathcal S$. Clearly, for such $S$, equality \eqref{itm:S1} holds.

Moreover, we claim that the set of $(\lambda,x)\in\Lambda\times S$ for which $T_xS+\mathcal D_x^\lambda=T_xX$ is open,
which implies easily that, by possibly taking a smaller $S$, also property \eqref{itm:S2} holds.
To prove the claim, use Lemma~\ref{thm:opensurjective} applied to the following pair of operators: $T_1=T_1(x,\lambda)=\mathrm d\beta_x^\lambda(1)\colon\mathfrak g_\lambda\to T_xX$,
and $T_2=T_2(x)\colon T_xS\to T_xX$ the inclusion of $T_xS$ into $T_xX$. By using suitable trivializations of the bundles $\mathfrak g=\bigcup_\lambda(\{\lambda\}\times\mathfrak g_\lambda)$
and $TX$ near $\lambda_0$ and $x_0$, clearly one can assume that the domains and the range of these operators are fixed Banach spaces.
Observe that $T_xS+\mathcal D_x^\lambda=T_xX$ is equivalent to $T_1(x,\lambda)+T_2(x)$ being surjective. On the other hand, the  kernel of $T_1(x,\lambda)+T_2(x)$
is always complemented, because it is finite-dimensional. Namely, $T_2(x)$ is injective, and the domain of $T_1(x,\lambda)$ is finite-dimensional.
Hence, Lemma~\ref{thm:opensurjective} implies that the set of $(\lambda,x)\in\Lambda\times S$ for which $T_xS+\mathcal D_x^\lambda=T_xX$ is open.

Property \eqref{itm:S3} holds if we show that $\beta_x^\lambda(G_\lambda)\cap S\ne\emptyset$, for $(\lambda,x)$ sufficiently close to $(\lambda_0,x_0)$. In the regular case that is being considered here, this follows easily by a transversality argument. Namely, by the continuity of
$\mathrm d\beta_x^\lambda$, property \eqref{itm:S1} implies that the map $\beta_x^\lambda$ is transversal to $S$ for $(\lambda,x)$ sufficiently close to $(\lambda_0,x_0)$.

Once the existence of a submanifold $S$ for which the above properties hold has been established, the proof of our result
is obtained from the following argument. Property \eqref{itm:S2} implies that, for $\lambda\in U$, the critical points of
the restriction $f_\lambda\big\vert_S\colon S\to\mathds R$ are in fact critical points of $f_\lambda$. By property~\eqref{itm:S3},
for all $\lambda\in U$ and all $y\in\widetilde V$, $\dd f_\lambda(y)=0$ if and only if
there exists $z\in S$, with $y\in G_\lambda\cdot x$, such that $\dd (f_\lambda\vert_S)(z)=0$.
Thus, the equivariant implicit function theorem is reduced to the standard implicit functions theorem for the restriction
$f_\lambda\vert_S$.

From here, the rest of the proof goes along the same line as the proof of the equivariant implicit function theorem
with fixed group of symmetries, as presented in \cite{BetPicSic1}. This part of the proof, that will be omitted here, uses
property~\eqref{itm:S1} of $S$, the gradient-like maps $\delta f_\lambda$, the vertical derivative of $\delta f_{\lambda_0}$ at $x_0$,
and the equivariant nondegeneracy assumption.
\end{proof}
\end{section}

\begin{section}{The non-regular case}
\label{sec:nonregular}
In several interesting situations arising from geometric variational problems, the existence of an implicit function
can only be obtained in a framework without the regularity properties assumed in Section~\ref{sec:proofregular}. We use the example of the constant mean curvature (CMC) variational problem to address some of those regularity issues, and then proceed to a case-by-case discussion (see Subsections~\ref{sub:nonreggroupactions}, \ref{subsec:actionhomeos}, \ref{subsec:ctsbundleliegps}) of how to weaken various hypotheses of the result proved in Section~\ref{sec:proofregular}.

In the CMC variational problem, one searches for minimizers of the area functional subject to a volume constraint, in the space of codimension $1$ \emph{unparametrized embeddings}\footnote{i.e., embeddings modulo reparametrizations.} (or immersions) of a compact manifold $M$ in a Riemannian manifold $\overline M$, see Subsection~\ref{subsec:cmc} for details. Equivariance in this scenario comes from the left-composition action of the isometry group of $\overline M$ on the space of embeddings of $M$ into $\overline M$. There are several technical issues concerning this problem:
\begin{itemize}
\item the space $X$ of $C^{k,\alpha}$ unparametrized embeddings of $M$ into $\overline M$ is not a differentiable manifold. It only admits a natural atlas of local charts, which are \emph{not} differentiably compatible;
\item although the isometry group of $\overline M$ is a Lie group, its action on $X$ is only continuous, and differentiable only on a dense subset (consisting of smooth, i.e., $C^\infty$, embeddings);
\item due to topological reasons, it may be convenient to not define the group action globally on $X$, but rather describe it as a \emph{local action}.
\end{itemize}
A detailed discussion of the first two points above can be found in \cite{AliPic10}.
The question of locality for the group action is basically a matter of technicalities, and dealing with this situation entails
no essential modification of the theory presented here. The interested reader may find a discussion of local actions
in the context of our equivariant implicit function theory in \cite{BetPicSic1}.
Let us simply emphasize here that an important consequence of extension to local actions is the corresponding
weakening on the assumption that $X$ admits a global differentiable structure. Namely, if $G$ is a group that acts
on $X$, then one has a local action of $G$ on any open subset of $X$. In particular, for the validity of the equivariant
implicit function theorem, it will be enough to have that $X$ admits an atlas of local charts that are only continuously
compatible, i.e., an atlas defining a global \emph{topological manifold} structure on $X$, provided that the relevant functions
be differentiable in each local chart.

This seemingly unusual situation actually occurs quite often in problems invariant under reparametrizations. In the case at hand, the space $X$ of $C^{k,\alpha}$ unparametrized embeddings of $M$ into $\overline M$ admits a local parametrization near $x_0\in X$ in terms of sections of the normal bundle of $x_0$; and it is well-known that two such local parametrizations are only continuously (and not differentiably) compatible. Nevertheless, smooth action functionals defined on the space of embeddings, that are invariant under diffeomorphisms of the source manifold, define continuous functions on the quotient space of unparametrized embeddings, that are also smooth in every local chart. This property is discussed thoroughly in \cite{AliPic10}.

\subsection{Non-regular group actions}
\label{sub:nonreggroupactions}
An important issue is the lack of differentiability of the maps $\beta_x^\lambda(g)=g\cdot x$, for all $x$. In typical geometric applications, $x$ is a given map of class $C^{k,\alpha}$ and $\beta^\lambda_x$ is differentiable (at $1_\lambda$) only when $x$ has higher regularity, say $C^\infty$. The set of such $x$ is in general dense, and the corresponding (densely defined) map $(x,\lambda)\mapsto\mathrm d\beta_x^\lambda(1_\lambda)$ admits a continuous extension to all $x$, as a section of an appropriate bundle. This will be discussed in concrete examples in the sequel; here we give abstract axioms to deal with this situation.

The regularity assumption for the action of $\mathcal G$ on $X$ in the equivariant implicit function theorem
can be replaced with the assumption that, for all $\lambda\in\Lambda$, the map $\beta_x^\lambda$ is differentiable at $1_\lambda$ for all
 $x$ belonging to a dense subset $X'\subset X$ containing $x_0$, and  assuming the existence of a $C^k$-vector bundle $\mathcal Y$ over $X$,
 together with $C^k$-bundle morphisms
$j\colon\mathcal E\to\mathcal Y^*$ and $\kappa\colon TX\to\mathcal Y$, satisfying the following properties:
\begin{itemize}
\item $\kappa$ is injective;
\item $\kappa^*\circ j\colon \mathcal E\to TX^*$ coincides with the inclusion $\mathcal E$ in $TX^*$ via the inner product
$\langle\cdot,\cdot\rangle$ (from which it follows that also $j$ must be injective);
\item the map $\Lambda\times X'\ni(\lambda,x)\mapsto\kappa_x\circ\,\dd\beta_x^\lambda(1_\lambda)\in\mathrm{Hom}(\mathfrak g_\lambda,\mathcal Y_x)$
has a continuous extension to a section of the vector bundle $\mathrm{Hom}(\mathrm{Lie}(\mathcal G),\mathcal Y)$ over $\Lambda\times X$.
\end{itemize}
With the above setup, one proves the existence of a submanifold $S$ of $X$ through $x_0$, such that $T_{x_0}S$
is a complement to $\mathcal D_{x_0}^{\lambda_0}$, and such that the standard implicit function theorem holds
for the restriction of $f_\lambda$ to $S$ around $(x_0,\lambda_0)$.
The proof of this fact is totally analogous to the proof of the same result in the case of fixed group of symmetries
discussed in \cite{BetPicSic1}, and will be omitted here. The argument is concluded by showing
the existence of a neighborhood
$\widetilde V\subset X$ of $x_0$ such that $G_\lambda\cdot S\supset\widetilde V$ for all $\lambda\in U$;
for this, the transversality argument used in the regular case cannot be used here.
Instead, the existence of such $\widetilde V$ is obtained
from Proposition~\ref{thm:nonemptyintersection_fiber}, applied to the following setup:
$A=\Lambda\times X$, $\mathcal M$ is the fiber bundle over $\Lambda\times X$ whose fiber over $(\lambda,x)$ is $G_\lambda$,
$a_0=(\lambda_0,x_0)$,
$p_0=\big((\lambda_0,x_0),1_{\lambda_0}\big)$, $N=X$, $Q=S$, and $\chi\big((\lambda,x),g\big)=\beta_x^\lambda(g)=g\cdot x$.
Assumption (b) of Proposition~\ref{thm:nonemptyintersection_fiber} is satisfied, since
by assumption $x_0\in X'$ and therefore $\chi\big((\lambda_0,x_0),\cdot\big)=\beta_{x_0}^{\lambda_0}$ is differentiable.
Assumption (c) holds because $T_{x_0}S$ is a complement to $\mathcal D_{x_0}^{\lambda_0}$.

\subsection{Actions by homeomorphisms}\label{subsec:actionhomeos}
In the proof of our equivariant implicit function theorem, we have used the assumption that the group actions
on $X$ are by diffeomorphisms in order to guarantee that, given $\lambda\in\Lambda$ and a critical point $x$ of
$f_\lambda$, then the $G_\lambda$-orbit of $x$ consists of critical points of $f_\lambda$. Namely, the proof
of this fact involves the derivative of the map $x\mapsto g\cdot x$, which in principle does not exist if
the action of $G_\lambda$ on $X$ is assumed to be only by homeomorphisms.
Thus, when dealing with actions by homeomorphisms, an explicit invariance property for critical points has to be assumed:
for all $\lambda$, the set $\big\{x\in X:\dd f_\lambda(x)=0\big\}$ is $G_\lambda$-invariant. Having such assumption satisfied,
the equivariant implicit function theorem holds also in the case of actions by homeomorphisms.

In practical terms, one obtains invariance of the set of critical points when the $G_\lambda$-action on $X$ lifts\footnote{i.e., there is a left-action of $G_\lambda$ on $\mathcal E$, such that the projection $\mathcal E\to X$ is equivariant with respect to such action.} to an action on the vector bundle $\mathcal E$ introduced in Section~\ref{sec:proofregular}, by linear isomorphisms on the fibers, and the gradient-like map $\delta f_\lambda\colon X\to\mathcal E$ is equivariant
with respect to such action. In such a situation, the null section of $\mathcal E$ is fixed by the $G_\lambda$-action. Thus, the set of critical points of $f_\lambda$, which coincides with the set of zeroes of the gradient-like map $\delta f_\lambda$, is $G_\lambda$-invariant.

\subsection{Actions of continuous bundles of Lie groups}\label{subsec:ctsbundleliegps}
While the existence of a Lie group structure on each symmetry group $G_\lambda$ is a central point in our theory, it may be interesting to weaken the regularity assumption on the dependence on the parameter $\lambda$.
More precisely, one can consider \emph{continuous}, rather than smooth, bundles of Lie groups (cf. Definition~\ref{def:smoothliegps}).

\begin{defin}\label{def:continuousliegps}
Let $\Lambda$ be a topological space, and, for all $\lambda\in\Lambda$, let $G_\lambda$ be a Lie group.
The set $\mathcal G=\bigcup_{\lambda\in\Lambda} (\{\lambda\}\times G_\lambda)$ is a \emph{continuous bundle of Lie groups} over $\Lambda$ if the following properties are satisfied:
\begin{itemize}
\item $\mathcal G$ is a groupoid with a topology whose restriction to each fiber $\{\lambda\}\times G_\lambda$ coincides with the topology of $G_\lambda$;
\item the composition operation $\mathcal G\times_\Lambda\mathcal G\ni (g,h)\to g\cdot h \in\mathcal G,$ is continuous;
\item the involution $i\colon \mathcal G\ni g\mapsto g^{-1}\in\mathcal G$ is continuous;
\item there exist \emph{local trivializations} of $\mathcal G$, in the following sense: for all $\lambda_0\in\Lambda$, there exists a neighborhood
$V$ of $\lambda_0$ in $\Lambda$, a neighborhood $U$ of $1_{\lambda_0}$ in $G_{\lambda_0}$,
a neighborhood $Z$ of $(\lambda_0,1_{\lambda_0})$ in $\mathcal G$  and a homeomorphism
$h\colon V\times U\to Z$ such that $h(\lambda_0,g)=(\lambda_0,g)$ for all $g\in U$ and for which the following diagram commutes:
\begin{equation}\label{eq:fibersontofibres}
\xymatrix@-5pt{V\times U\ar[rr]^h\ar[dr]_{\pi_1}&&Z\ar[ld]^\pi\cr&V}
\end{equation}
where $\pi_1\colon V\times U\to V$ is the projection onto the first variable, and $\pi\colon\mathcal G\to\Lambda$ is the natural projection;
\item commutativity of \eqref{eq:fibersontofibres} means that $h$ carries \emph{fibers} $\{\lambda\}\times U\subset V\times G_{\lambda_0}$ onto fibers $G_\lambda\cap Z\subset\mathcal G$. It is required that the restriction of $h$ to each fiber be differentiable, and
that the fiber derivative $\partial_2h\colon U\times TG_{\lambda_0}\to T\mathcal G$ be a continuous map.
\end{itemize}
\end{defin}

\begin{rem}
It is not hard to prove that a \emph{smooth} bundle of Lie groups, as in Definition~\ref{def:smoothliegps}, is also a
\emph{continuous} bundle of Lie groups. The existence of local trivializations for a smooth bundle of Lie groups
is obtained easily using standard results regarding the local form of submersions, applied to the source map
$\mathfrak s\colon\mathcal G\to\Lambda$.
\end{rem}

The above axioms imply that the identity section $s_\mathrm I\colon \Lambda\to\mathcal G$, $s_\mathrm I(\lambda)=1_\lambda$ for all $\lambda\in\Lambda$, is continuous. Namely, for the continuity at any fixed $\lambda_0$, consider a local trivialization $h\colon V\times U\to Z$ as above, and set $g_\lambda=h(\lambda,1_{\lambda_0})$ for all $\lambda\in V$. Then, $s_\mathrm I(\lambda)=g_\lambda^{-1}\cdot h(\lambda,1_{\lambda_0})$, which is continuous by the above axioms.

It also follows easily that the set $\mathrm{Lie}(\mathcal G):=\bigcup_{\lambda\in\Lambda}(\{\lambda\}\times\mathfrak g_\lambda)$
is a topological vector bundle over $\Lambda$. Namely, given a local trivialization $h\colon V\times U\to Z$ as
in \eqref{eq:fibersontofibres}, then the map $\widetilde h\colon V\times U\to\widetilde Z$ defined by $\widetilde h(\lambda,g)=h(\lambda,1_{\lambda_0})^{-1}h(\lambda,g)$
is another local trivialization that satisfies $\widetilde h(\lambda,1_{\lambda_0})=1_\lambda$ for all $\lambda\in V$.
Then, vector bundle local trivializations for $\mathrm{Lie}(\mathcal G)$ are obtained by
differentiating at $1_{\lambda_0}$ such a map $\widetilde h$.

The continuity of the action of $\mathcal G$ on $X$ is intended in the sense that the map $\beta\colon\mathcal G\times X\to X$, defined by $\beta(\lambda,g,x)=\beta_x^\lambda(g)=g\cdot x$, is continuous in all three variables. The minimal regularity requirements for the action are that
\begin{itemize}
\item there exists a dense subset $X'\subset X$ such that $\beta(\lambda,\cdot,x)$ is differentiable at $1_\lambda$ for all $\lambda\in\Lambda$ and all $x\in X'$;
\item as in Subsection~\ref{sub:nonreggroupactions}, given vector bundles $\mathcal E$ and $\mathcal Y$, with morphisms
$j\colon \mathcal E\to\mathcal Y^*$ and $\kappa\colon TX\to\mathcal Y$,
the map $k_x\circ\mathrm d\beta(\lambda,1_\lambda,x)$ must admit a continuous extension to a global section
of the bundle $\mathrm{Hom}(\mathrm{Lie}(\mathcal G),\mathcal Y)$.
\end{itemize}
Equipped with the above topological framework, the proof of the equivariant implicit function theorem in Section~\ref{sec:proofregular} carries over \emph{verbatim} to the case of continuous bundles of Lie groups.
\end{section}

\begin{section}{Applications to Geometric Variational Problems}
\label{sec:applications}

In this section, we discuss various applications of the above abstract deformation results to geometric variational problems. The necessity of the non-regular extensions is illustrated in two of the three applications discussed, namely the variational problems of finding CMC and Hamiltonian stationary Lagrangian submanifolds. In these two problems, invariance under reparametrizations causes the group action to be not differentiable at every point. The third application, concerning harmonic maps between manifolds, gives an example of the regular case, where these extensions are not needed.

\subsection{Deformation of CMC hypersurfaces}\label{subsec:cmc}
Let $M$ be a compact manifold, with $\dim M=m$, and let $\overline M$ be a differentiable manifold,
with $\dim\overline M=m+1$, endowed  with a Riemannian metric $\overline{\mathbf g}$.
For $\alpha\in\left]0,1\right[$, denote by $\mathrm{Emb}^{2,\alpha}(M,\overline M)$
the set of all embeddings of class $C^{2,\alpha}$ of $M$ into $\overline M$, which is an open subset of the Banach space
$C^{2,\alpha}(M,\overline M)$ of all $C^{2,\alpha}$-functions from $M$ to $\overline M$. Two embeddings $x_i\colon M\to\overline M$, $i=1,2,$ are said to be \emph{isometrically congruent} if there exists a diffeomorphism $\phi$ of $M$ and an isometry $F$ of $\overline M$ such that $x_2=F\circ x_1\circ\phi$.
Recall that the mean curvature $\mathcal H_x$ of an embedding $x\colon M\to\overline M$ is the norm of the \emph{mean curvature vector} $\vec H_x$ of $x$, which is the trace of the second fundamental form $\mathcal S_x$. The embedding is said to have \emph{constant mean curvature $H$} (in short, CMC), if $\mathcal H_x=H$.

The CMC embedding problem has a convenient variational framework that we now describe.
Suppose $x\colon M\to\overline M$ is a CMC embedding such that $x(M)$ is the boundary of a bounded domain of $\overline M$. Let $X$ be the (component of $x$ of the) set of \emph{$C^{2,\alpha}$-unparametrized embeddings} of $M$ into $\overline M$, i.e., the quotient of (component containing $x$ of) the space of $C^{2,\alpha}$-embeddings of $M$ into $\overline M$ by the action of the diffeomorphism group of $M$. In other words, $X$ is the
set of $C^{2,\alpha}$-submanifolds of $\overline M$ that are diffeomorphic to $M$.
Such a space admits a global topological manifold
structure, modeled on the Banach space $C^{2,\alpha}(M)$ of real-valued functions on $M$ of class $C^{2,\alpha}$. Given a smooth embedding $x\colon M\to\overline M$, denote by $[x]\in X$ the corresponding class. Unparametrized embeddings in a neighborhood of $[x]\in X$ are parametrized by $C^{2,\alpha}$-sections $V$ of the normal bundle of $x$ near the null section. Given one such $V$, the associated unparametrized embedding is the class of the embedding $M\ni p\mapsto\exp_{x(p)}V_p\in\overline M$, where $\exp$ is the exponential map of $(\overline M,\overline{\mathbf g})$. Since $x(M)$ is the boundary of a domain in $\overline M$, the normal bundle of $x(M)$ is orientable. Thus, its sections are of the form $V=h\,\vec{n}_x$, where $\vec{n}_x$ is the unit normal field of $x$, and such section $V$ is identified with the function $h\colon M\to\R$. In particular, for all $x\in X$, the tangent space $T_{[x]}X$ is identified with $C^{2,\alpha}(M)$.

Let $\overline{\mathbf g}_\lambda$ be a smooth family of metrics on $\overline M$, parametrized by $\lambda\in\Lambda$, where $\Lambda$ is a smooth manifold. Given an unparametrized embedding $[x]$, where the image of $x$ is the boundary of a bounded domain $\Omega_x\subset\overline M$, denote by $\area_\lambda(x)$ the volume of $(M,\mathbf g_\lambda)$, where $\mathbf g_\lambda$ is the pull-back metric $x^*(\overline{\mathbf g}_\lambda)$. Denote by $\Vol_\lambda(x)$ the volume of $\Omega_x$ relatively to the volume form of $\overline{\mathbf g}_\lambda$. For a given $H\in\R$, set
\begin{equation}\label{eq:functionalcmc}
f\colon X\times\Lambda\to\R, \quad f([x],\lambda)=\area_\lambda(x)+H\Vol_\lambda(x).
\end{equation}
Regularity properties of this functional in the space of unparametrized embeddings are discussed in~\cite{AliPic10}. If $x$ is a smooth embedding, up to using a local chart around $[x]\in X$ as an identification, we can consider the first variation of the above functional with respect to a variation $V=h\,\vec{n}_x$, which is well-known to be
\begin{equation}\label{eq:firstvarcmc}
\partial_1 f([x],\lambda)V=\int_M(H-\mathcal H_x)h \vol_{\mathbf g_\lambda},
\end{equation}
where $\mathcal H_x$ is the mean curvature function of the embedding $x\colon M\to\overline M$. As a result, critical points of $f_\lambda\colon X\to\R$ are precisely the classes of smooth embeddings of $M$ into $\overline M$
that have constant mean curvature equal to $H$, when $\overline M$ is endowed with the metric $\overline{\mathbf g}_\lambda$. The corresponding \emph{Jacobi operator} $J_x$ of a CMC embedding $x$ is the linear elliptic differential operator
\begin{equation}\label{eq:JacobiCMC}
J_x(h)=\Delta_x h-\left(\Ric(\vec{n}_x)+\|\mathcal S_x\|^2\right)h,
\end{equation}
where $\Delta_x$ is the (positive) Laplacian of functions on $M$ relative to the metric $x^*(\overline{\mathbf g}_\lambda)$, $\Ric(\vec{n}_x)$ is the Ricci curvature of $\overline M$ evaluated on the unit normal $\vec{n}_x$ and $\mathcal S_x$ is the second fundamental form of $x$. A \emph{Jacobi field} along $x$ is a smooth function $h\colon M\to\mathds R$ satisfying $J_x(h)=0$.

Let $G_\lambda=\Iso(\overline M,\overline{\mathbf g}_\lambda)$ be the isometry group of $(\overline M,\overline{\mathbf g}_\lambda)$, which is a finite-dimensional Lie group. For each $\lambda\in\Lambda$, the Lie group $G_\lambda$ acts on $X$ by left-composition $g\cdot [x]=[g\circ x]$. This action is \emph{continuous}, but \emph{not differentiable} on all of $X$. Both $\area_\lambda$ and $\Vol_\lambda$ on \eqref{eq:functionalcmc} are $G_\lambda$-invariant functions on $X$, and hence so is $f_\lambda$. This invariance implies that symmetries of the ambient, encoded in the form of Killing fields, induce certain Jacobi fields on every CMC embedding.

\begin{defin}
Given a Killing vector field $K$ on $(\overline M,\mathbf g_\lambda)$, the function $h_K=\overline{\mathbf g}_\lambda(K,\vec n_x)$ is a Jacobi field. Jacobi fields of this form are called \emph{Killing-Jacobi fields}, and $h_K$ is said to be \emph{induced by $K$}. The CMC embedding $x$ is said to be \emph{equivariantly nondegenerate} if all the Jacobi fields along $x$ are Killing-Jacobi fields.
\end{defin}

In other words, $x$ is equivariantly nondegenerate if all infinitesimal CMC deformations of $x$ with the same mean curvature arise from isometries of the ambient space. Equivariant nondegeneracy of CMC embedding is the key property to obtain the following deformation results for CMC embeddings:

\begin{teo}\label{thm:applCMC}
Let $\overline{\mathbf g}_\lambda$ be a family of Riemannian metrics on $\overline M$, parametrized by $\lambda\in\Lambda$, where $\Lambda$ is a smooth manifold. Let $x_0$ be an equivariantly nondegenerate embedding of $M$ into $(\overline M,\overline{\mathbf g}_{\lambda_0})$ with constant mean curvature $H$, such that $x_0(M)$ is the boundary of a bounded domain of $\overline M$.
Assume that $G_\lambda=\mathrm{Iso}(\overline M,\mathbf g_\lambda)$, $\lambda\in\Lambda$, forms a continuous bundle
of Lie groups $\mathcal G=\{G_\lambda:\lambda\in\Lambda\}$.
Then, there exists a neighborhood $V$ of $\lambda_0$ in $\Lambda$, and a smooth function $x\colon V\to\mathrm{Emb}^{2,\alpha}(M,\overline M)$
satisfying:
\begin{itemize}
\item[(a)] $x(\lambda)$ is an embedding of $M$ into $(\overline M,\overline{\mathbf g}_\lambda)$ with constant mean curvature $H$ for all $\lambda\in V$;
\item[(b)] $x(\lambda_0)=x_0$.
\end{itemize}
Moreover, given $\lambda\in V$, any other embedding of constant mean curvature $H$ of $M$ into $(\overline M,\overline{\mathbf g}_\lambda)$ sufficiently close to $x(\lambda)$ in $\mathrm{Emb}^{2,\alpha}(M,\overline M)$ is isometrically congruent to $x(\lambda)$.
\end{teo}

\begin{proof}
This  is an application of our equivariant implicit function theorem to the above variational setup that requires all the non-regularity extensions discussed in
Section~\ref{sec:nonregular}. As indicated above, $X$ is the space of unparametrized embeddings of $M$ into $\overline M$ of class $C^{2,\alpha}$, and $f$ is given by \eqref{eq:functionalcmc}. For all $\lambda\in\Lambda$, the Lie group $G_\lambda$ acts on $X$ by left-composition, and the maps $\beta_{[x]}^\lambda\colon G_\lambda\to X$ are smooth when $[x]$ is the class of a smooth embedding. The set $X'$ of such embeddings is dense in $X$. In particular, the image of $\mathrm d\beta_{[x]}^\lambda(1_\lambda)$ coincides with the space of Killing-Jacobi fields along $x$, so that the equivariant nondegeneracy assumption on $x$ is precisely the equivariant nondegeneracy assumption of our implicit function theorem.

All the other objects involved in the statement and the proof of the abstract implicit function theorem, described in Section~\ref{sec:proofregular} and Section~\ref{sec:nonregular}, are as follows:
\begin{itemize}
\item $\mathcal E$ is the vector bundle with fiber over $[x]$ the Banach space of  $C^{0,\alpha}$-sections of
the normal bundle of $x$. Recall that $T_{[x]}X$ is the Banach space of $C^{2,\alpha}$-sections of
the normal bundle of $x$, and the inclusion $TX\subset\mathcal E$ is obvious.
\item The inner product $\langle\cdot,\cdot\rangle_{[x]}$ on $\mathcal E_{[x]}$ is the $L^2$-paring $\langle h_1,h_2\rangle_{[x]}=\int_M h_1\, h_2\,\mathrm{vol}_{\mathbf g_*}$ with respect to the volume form of a background metric\footnote{The Riemannian metric $\mathbf g_*$ is a (fixed) reference metric, an can be chosen arbitrarily. It simply serves the purpose of inducing an $L^2$-pairing that does not vary with $\lambda$.} $\mathbf g_*$;
\item $\mathcal Y$ is the vector bundle on $X$, whose fiber at the point $[x]$ is the Banach space of $C^{1,\alpha}$-sections of the normal bundle of $x$. The bundle morphism $\kappa\colon TX\to\mathcal Y$ is the obvious inclusion, and the morphism $j\colon \mathcal E\to\mathcal Y^*$ is induced by the $L^2$-pairing above;
\item identifying the Lie algebra $\mathfrak g_\lambda$ with the space of (complete) Killing vector fields
on $(\overline M,\overline{\mathbf g}_\lambda)$, for $[y]\in X'$, the map $\dd\beta_{[y]}^\lambda(1_\lambda)\colon\mathfrak g_\lambda\to T_{[y]}X$ associates to a Killing vector field ${K}$ the $\overline{\mathbf g}_\lambda$-orthogonal component of ${K}$ along $y$;
\item $\delta f_\lambda([x])=(H-\mathcal H_x)\,\xi_{\mathbf g_\lambda}$, where $\mathcal H_x$ is the mean curvature function of $x$ (which is a $C^{0,\alpha}$-function on $M$) and $\xi_{\mathbf g_\lambda}$ is the positive function satisfying $\xi_{\mathbf g_\lambda}\vol_{\mathbf g_*}=\vol_{\mathbf g_\lambda}$, see \eqref{eq:firstvarcmc};
\item the vertical derivative of $\delta f_\lambda$ at $[x]$ is identified with the Jacobi operator $J_x$, as indicated in \eqref{eq:ddfjacobi}. As a linear elliptic operator of second-order, this is a Fredholm operator of index $0$ from $C^{2,\alpha}(M)$ to $C^{0,\alpha}(M)$.\qedhere
\end{itemize}
\end{proof}

\subsection{More general results on CMC deformations}
In view to concrete applications, it is worth discussing a few generalizations of Theorem~\ref{thm:applCMC} that allow to obtain deformation results for CMC \emph{immersions}, and not only \emph{embeddings}.

\subsubsection{Generalized volume functionals}
Let us start with some remarks on the assumption that $x(M)$ should be the boundary of a domain in $\overline M$. First, it must be observed that Theorem~\ref{thm:applCMC} does not hold in full generality if such assumption is dropped, see \cite[Ex ~4.11]{BetPicSic1}. In particular, with the above statement, the result of Theorem~\ref{thm:applCMC} does not apply to CMC embeddings of manifolds with boundary.
There exist, however, interesting generalizations of the result obtained by weakening this assumption. The crucial observation is that such assumption is needed for the variational characterization of CMC embeddings, which uses the volume functional $\Vol_\lambda$. For the validity of the result, one may assume more generally the existence of an \emph{invariant volume functional} in a neighborhood of $x$ in the space of embeddings of $M$ into $\overline M$. Details of the definition of invariant volume functionals are discussed in \cite[App~B]{BetPicSic1}. Let us simply list some conditions under which invariant volume functionals exist:
\begin{itemize}
\item $\overline M$ noncompact and $\mathrm{Iso}(\overline M,\overline{\mathbf g}_0)$ compact;
\item $\partial M\neq\emptyset$ and $\overline M$ noncompact;
\item $\overline M$ noncompact with $H^m(\overline M,\R)=0$, or, more generally, if the image of $x$ is contained in an open subset of $\overline M$ whose $m^{th}$ de Rham cohomology vanishes.
\end{itemize}
The last item includes, for instance, the cases $\overline M=\mathds R^{m+1}$ and $S^{m+1}$. Moreover, manifolds of the form $\overline M^{m+1}=\mathds R^k\times N^{m+1-k}$, $k\ge1$, have trivial $m^{th}$ de Rham cohomology; and manifolds of the form $\overline M^{m+1}=S^k\times N^{m+1-k}$, $k\ge1$, have open dense subsets with trivial $m^{th}$ de Rham cohomology.

\subsubsection{Free immersions}
As we have pointed out, the space of unparameterized embeddings of a manifold $M$ into another manifold $\overline M$
is described as the space of orbits of the action of the diffeomorphism group of $M$ acting by right-composition in the
space of embeddings of $M$ into $\overline M$. Such action is free, and this fact plays an important role in establishing
the (local) smooth structure for the space of unparameterized embeddings. A general immersion $x\colon M\to\overline M$ can
have non-trivial stabilizer in the group of diffeomorphisms; for instance, an iterated closed
curve in $\overline M$ is an immersion of $S^1$ whose stabilizer is a finite cyclic group.

The variational theory for parametrization invariant geometric functionals can be extended to the case of immersions
whose stabilizer is trivial, i.e., the so-called \emph{free immersions} of $M$ into $\overline M$. More precisely, an immersion $x\colon M\to\overline M$ is \emph{free} if the only diffeomorphism $\varphi\colon M\to M$ such that $x\circ\varphi=x$ is the identity. The differentiable structure of free immersions is studied in detailed, for the smooth case, in \cite{CerMasMic91}. The appropriate extensions to the $C^{2,\alpha}$-case, needed for dealing with the CMC variational problem, can be carried out as in \cite{AliPic10}. We point out that being free is a rather weak assumption; for instance, if there exists at least one point $p$ such that the inverse image of $x(p)$ consists of a single point, then the immersion $x$ is free (see \cite[Lemma~1.4]{{CerMasMic91}}).
Theorem~\ref{thm:applCMC} holds, \emph{mutatis mutandis}, under the more general assumptions that
$x_0\colon M\to\overline M$ is a free immersion. A general statement including this case will be given below
(Theorem~\ref{thm:applCMC_ext}).

\subsubsection{A generalized equivariant nondegeneracy assumption}
As a motivation for the introduction of a more general nondegeneracy assumption in Theorem~\ref{thm:applCMC},
let us observe that great attention has been given recently to the study of submanifold theory (surfaces) in ambient spaces like $M_k=M^2(k)\times\mathds R$, where $M^2(k)$ is a $2$-dimensional space form of curvature $k$, see \cite{bd,hw1,hw2,mr,mora}. In this case, when $k\ne0$, the isometry group of $M_k$ has dimension $4$, while for $k=0$, since $M_0$ is the standard Euclidean $3$-space, the isometry group is $6$-dimensional. Thus, if one wants to prove deformation results for surfaces in $\mathds R^3$ to surfaces in $M_k$, $k\ne0$, Theorem~\ref{thm:applCMC} cannot be applied.

Nonetheless, we observe that the isometry group of $\mathds R^3$ contains a closed subgroup $G_0$
such that, denoting $G_k=\Iso(M_k)$, $k\ne0$, the set
$\big\{(k,G_k):k\in\mathds R\big\}$ is a smooth bundle of Lie groups. Namely, it is enough to define $G_0$ as the group
of isometries of $\mathds R^3$ that preserve the splitting $\mathds R^2\times\mathds R$. Clearly,
such group is isomorphic to the product $\Iso(\mathds R^2)\times \mathds R$, and has dimension $4$. This motivates the following:

\begin{defin}
Let $(\overline M,\mathbf g)$ be a Riemannian manifold, $\Iso(\overline M,\mathbf g)$ the group of isometries of $(\overline M,\mathbf g)$,
$\mathfrak{Iso}(\overline M,\mathbf g)$ its Lie algebra, and let $\mathfrak K$ be a Lie subalgebra of $\mathfrak{Iso}(\overline M,\mathbf g)$.
A codimension $1$ CMC embedding $x\colon M\to\overline M$ is said to be \emph{$\mathfrak K$-nondegenerate} if every Jacobi field along $x$ is a Killing-Jacobi field induced by some element of $\mathfrak K$.
\end{defin}

In conclusion, we can formulate the following extension of Theorem~\ref{thm:applCMC}:

\begin{teo}\label{thm:applCMC_ext}
Let $M^m$ be a compact manifold (possibly with boundary) and let $\overline M^{m+1}$ be any differentiable manifold.
Let $\overline{\mathbf g}_\lambda$ be a family of Riemannian metrics on $\overline M$, parameterized by $\lambda\in\Lambda$,
where $\Lambda$ is a smooth manifold. Let $\lambda_0\in\Lambda$ be fixed; for $\lambda\ne\lambda_0$, set $G_\lambda=\mathrm{Iso}(\overline M,\mathbf g_\lambda)$,
and let $G_{\lambda_0}$ be a closed subgroup of $\Iso(\overline M,\mathbf g_{\lambda_0})$, with Lie algebra $\mathfrak g_{\lambda_0}$. Let $x_0$ be a free immersion of $M$ into $(\overline M,\overline{\mathbf g}_{\lambda_0})$ having constant mean curvature $H$.
 Assume the following:
 \begin{itemize}
 \item there exists an invariant volume functional in a neighborhood of $x_0(M)$;
 \item the family $G_\lambda$, $\lambda\in\Lambda$, forms a continuous bundle of Lie groups;
 \item $x_0$ is $\mathfrak g_{\lambda_0}$-nondegenerate.
 \end{itemize}
Then, there exists a neighborhood $V$ of $\lambda_0$ in $\Lambda$, and a smooth function $x\colon V\to C^{2,\alpha}(M,\overline M)$
satisfying:
\begin{itemize}
\item[(a)] $x(\lambda)$ is an immersion of $M$ into $(\overline M,\overline{\mathbf g}_\lambda)$ with constant mean curvature $H$, and
with $x(\lambda)(\partial M)=x_0(\partial M)$ for all $\lambda\in V$;
\item[(b)] $x(\lambda_0)=x_0$.
\end{itemize}
Moreover, given $\lambda\in V$, any other fixed boundary immersion of constant mean curvature $H$ of $M$ into $(\overline M,\overline{\mathbf g}_\lambda)$, sufficiently close to $x(\lambda)$ in $C^{2,\alpha}(M,\overline M)$ is isometrically congruent to $x(\lambda)$.
\end{teo}

Examples where Theorem~\ref{thm:applCMC_ext} is applied can be obtained by looking at
CMC immersions of manifolds with boundary conditions.
\begin{example}\label{ex:fixedboundaryCMC}
Let $M$ be a compact manifold with boundary, with, say $\partial M=N_1\cup N_2$, where $N_1$ and $N_2$ the two connected
components of $\partial M$. Let $x\colon M\to\mathds R^3$ be a CMC immersion such that $x(N_1)$ and $x(N_2)$ belong to two parallel horizontal planes, say $x_3=a_i$, $i=1,2$.
Here, we think of $\mathds R^3$ with coordinates $(x_1,x_2,x_3)$, so $x_3=\text{const.}$ are the horizontal planes. We say that $x$ is equivariantly nondegenerate if the only
Jacobi fields $J$ along $x$ that are horizontal on $\partial M$ are Killing-Jacobi fields. For $k\in\mathds R$, let $M^2(k)$ be the
$2$-dimensional space form of curvature $k$. Note that equivariant nondegeneracy does not imply that $x$ is rotationally invariant.

If $x$ is equivariantly nondegenerate, then $x$ can be smoothly deformed
to CMC immersions of $M$ into the product $M^2(k)\times\mathds R$, with the same constant mean curvature, and with
boundary on the slices $M^2(k)\times\{a_i\}$, $i=1,2$. This follows from Theorem~\ref{thm:applCMC_ext}, observing that any Killing vector field in $\mathds R^3$ whose restriction to a nontrivial closed horizontal curve is horizontal must be the sum of an infinitesimal horizontal rotation and a horizontal translation.\footnote{More generally, any Killing vector field $K$ in $\mathds R^3$ which is horizontal at three noncollinear points lying in a horizontal plane,
must be the sum of an infinitesimal horizontal rotation and a horizontal translation. Write a general Killing field in $\mathds R^3$ as $K=D+W$, where $D\in\mathfrak{so}(3)$ is an infinitesimal rotation and $W\in\mathds R^3$ is a translation. Assume $P_1$, $P_2$ and $P_3$ three noncollinear points on the plane $z=z_0$, and set $v_1=\overrightarrow{P_1P_3}$, $v_2=\overrightarrow{P_2P_3}$;
these are linearly independent horizontal vectors. Then, $(Dv_1)\cdot e_3=(Dv_2)\cdot e_3=0$, i.e., $D$ is a horizontal infinitesimal rotation.
From this, it also follows that $W$ is horizontal.}
In other words, Theorem~\ref{thm:applCMC_ext} is applied to the Lie algebra $\mathfrak g_0$ consisting of vector
fields that generate isometries of $\mathds R^3$ that preserve the splitting $\mathds R^2\times\mathds R$. In particular, the result applies to every equivariantly nondegenerate portion of nodoid or unduloid in $\mathds R^3$, with boundary
on horizontal planes.
\end{example}

\subsection{Deformation of harmonic maps}
Let $(M,\mathbf g)$ and $(\overline M,\overline{\mathbf g})$ be Riemannian manifolds.
A $C^2$-map $\phi\colon M\to\overline M$
is said to be $(\mathbf g,\overline{\mathbf g})$-\emph{harmonic} if
\begin{equation}\label{eq:deflaplacian}
\Delta_{\mathbf g,\overline{\mathbf g}}(\phi):=
\mathrm{tr}(\widehat{\nabla}\mathrm d\phi)=0
\end{equation}
where $\widehat{\nabla}$ is the
connection on the vector bundle $TM^*\otimes\phi^*(T\overline{M})$ over $M$
induced by the Levi-Civita connections $\nabla$ of $\mathbf g$ and $\overline{\nabla}$ of $\overline{\mathbf g}$.
When the source manifold $M$ is compact, harmonic maps on $M$ have the following variational characterization.

Let $X$ be the Banach manifold $C^{2,\alpha}(M,\overline M)$ of maps $\phi\colon M\to\overline M$ of class $C^{2,\alpha}$. Let $\Lambda$ be a smooth manifold that parametrizes a family of pairs $({\mathbf g}_\lambda,{\overline{\mathbf g}}_\lambda)$, where ${\mathbf g}_\lambda$ is a metric on $M$ and $\overline{\mathbf g}_\lambda$ is a metric on $\overline M$, both of class $C^k$ for a fixed $k\ge3$. Set
\begin{equation}\label{eq:intaction}
f\colon X\times\Lambda\to\R, \quad f(\phi,\lambda)=\tfrac12\int_M\Vert\mathrm d\phi(x)\Vert_{HS}^2\,\mathrm{vol}_{{\mathbf g}_\lambda},
\end{equation}
where $\mathrm{vol}_{\mathbf g_\lambda}$ is the volume form (or density) of $\mathbf g_\lambda$ and $\Vert\mathrm d\phi(x)\Vert_{HS}$ is the Hilbert-Schmidt norm of the linear map $\mathrm d\phi(x)$.
Note that the dependence on $\overline{\mathbf g}_\lambda$ in the above formula is hidden in $\Vert\mathrm d\phi(x)\Vert_{HS}$, which depends on the metrics of both source and target manifolds.

For a given $\lambda\in\Lambda$, critical points of $f_\lambda(\phi)= f(\phi,\lambda)$ are precisely
the $(\mathbf g_\lambda,\overline{\mathbf g}_\lambda)$-harmonic maps $\phi\colon M\to\overline M$.
For $\phi\in X$, the tangent space $T_\phi X$ is identified
with the Banach space $C^{2,\alpha}(\phi^*T\overline{M})$ of all vector fields along $\phi$ of class $C^{2,\alpha}$. Given a such $V\in T_\phi X$,
the first variation of $f$ in this direction is given by:
\allowdisplaybreaks
\begin{equation}\label{eq:primavarharm}
\partial_1f(\phi,\lambda)V=
\int_M \mathrm{tr}\left(\mathrm d\phi^*\overline{\nabla}V\right) \,\mathrm{vol}_{\mathbf g_\lambda} =
-\int_M\overline{\mathbf g}_\lambda\big(\Delta_{\mathbf g_\lambda,\overline{\mathbf g}_\lambda}(\phi),V\big)\,\mathrm{vol}_{\mathbf g_\lambda},
\end{equation}
where the trace is meant on the entries $\mathrm d\phi^*(\cdot)\overline{\nabla}_{(\cdot)}V$.
The correspondent
{\em Jacobi operator} $J$ along a $(\mathbf g_\lambda,\overline{\mathbf g}_\lambda)$-harmonic map $\phi$ is the linear differential operator:
\begin{equation}\label{eq:harmjacobiop}
J_\phi(V)=-\Delta V+\mathrm{tr}\left(\overline{R}(\mathrm d\phi(\cdot),V)\mathrm d\phi(\cdot)\right),
\end{equation}
defined in $C^{2,\alpha}(\phi^*T\overline{M})$. Here $\overline R$ is the
curvature tensor of $\overline{\mathbf g}_\lambda$, and
$\Delta V$ is a vector field along $\phi$ uniquely defined by
\begin{equation}\label{eq:lapv}
\overline{\mathbf g}_\lambda(\Delta V,W)=\mathrm{div}(\overline{\nabla}V^*)W -\overline{\mathbf g}_\lambda(\overline{\nabla}V,\overline{\nabla}W),
\quad W\in C^{2,\alpha}(\phi^*T\overline{M})
\end{equation}
i.e., $\Delta V(x)=\sum_i\left(\overline{\nabla}_{e_i}\overline{\nabla}V\right)e_i$,
where $(e_i)_i$ is an orthonormal basis of $T_xM$. In this situation, a vector field $V$ along $\phi$ that satisfies $J_\phi(V)=0$ is called a {\em Jacobi field}.
There are two special types of Jacobi fields along a $(\mathbf g_\lambda,\overline{\mathbf g}_\lambda)$-harmonic map $\phi$.

Let $G_\lambda$ be the product of isometry groups $\mathrm{Iso}(M,{\mathbf g}_\lambda)\times\mathrm{Iso}(\overline M,\overline{\mathbf g}_\lambda)$;
there is a natural action of $G_\lambda$ on $X$ given by
\begin{equation}
G_\lambda\times X\ni\big((g,\overline g),\phi\big)\mapsto\overline g\circ\phi\circ g^{-1}\in X.
\end{equation}
Clearly, \eqref{eq:intaction} is invariant under this action. Using results from \cite{palais}, one can prove that this action is smooth, since it is given by left and right-composition with smooth maps (see \cite{PicTau} for the noncompact case), and no reparametrization issue is present. Consequently, part of the technical low regularity arguments in our equivariant implicit function theorem are not required here. Due to these symmetries, if $\overline K$ is a Killing vector field of $(\overline M,\overline{\mathbf g}_\lambda)$, then $\overline K\circ\phi$ is a Jacobi field. Likewise, if $K$ is a Killing field of $(M,\mathbf g_\lambda)$, then $\phi_*(K)$,
defined by $M\ni p\mapsto\mathrm d\phi(p)K_p\in T_{\phi(p)}\overline M$, is a Jacobi field.

\begin{defin}
The space of Jacobi fields along $\phi$ spanned by fields of the type $\overline K\circ\phi$ and
$\phi_*(K)$, for Killing fields $\overline K$ and $K$, is called the space of \emph{Killing-Jacobi fields along $\phi$}.
A $(\mathbf g_\lambda,\overline{\mathbf g}_\lambda)$-harmonic map $\phi\colon M\to\overline M$
is said to be \emph{equivariantly nondegenerate} if the space of Jacobi fields along $\phi$ coincides with the space of Killing-Jacobi fields along $\phi$.
\end{defin}

\begin{defin}
Two harmonic maps $\phi_1$ and $\phi_2$ are called {\em geometrically equivalent} if
they are in the same $\mathrm{Iso}(M,{\mathbf g}_\lambda)\times\mathrm{Iso}(\overline M,\overline{\mathbf g}_\lambda)$-orbit, i.e.,  if there
exists $g\in\mathrm{Iso}(M,{\mathbf g}_\lambda)$ and $\overline g\in\mathrm{Iso}(\overline M,\overline{\mathbf g}_\lambda)$ such that $\phi_2=\overline g\circ\phi_1\circ g^{-1}$.
\end{defin}

\begin{teo}
Let $({\mathbf g}_\lambda,\overline{\mathbf g}_\lambda)$ be a family of pairs of Riemannian metrics on $M$ and $\overline M$ respectively, parametrized by $\lambda\in\Lambda$, for some manifold $\Lambda$. Assume that $G_\lambda=\mathrm{Iso}(M,{\mathbf g}_\lambda)\times\mathrm{Iso}(\overline M,\overline{\mathbf g}_\lambda)$, $\lambda\in\Lambda$, forms a smooth bundle of Lie groups $\mathcal G=\{G_\lambda:\lambda\in\Lambda\}$.
Let $\phi_0\colon M\to\overline M$ be an equivariantly nondegenerate $({\mathbf g}_{\lambda_0},\overline{\mathbf g}_{\lambda_0})$-harmonic map.
 Then, there exists a neighborhood $V$ of $\lambda_0$ in $\Lambda$, and a smooth
function $\phi\colon V\to X$ satisfying:
\begin{itemize}
\item[(a)] $\phi(\lambda)$ is a $({\mathbf g}_\lambda,\overline{\mathbf g}_\lambda)$-harmonic map for all $\lambda\in V$;
\item[(b)] $\phi(\lambda_0)=\phi_0$.
\end{itemize}
Moreover, given $\lambda\in V$, any other $({\mathbf g}_\lambda,\overline{\mathbf g}_\lambda)$-harmonic map from $M$ to $\overline M$
sufficiently close to $\phi_0$ is geometrically equivalent to $\phi(\lambda)$.
\end{teo}

\begin{proof}
All assumptions of the equivariant implicit function theorem are satisfied by the
harmonic maps variational problem, using the following objects:
\begin{itemize}
\item $X'=X=C^{2,\alpha}(M,\overline M)$;

\item $\mathcal E$ is the mixed vector bundle whose fiber $\mathcal E_\phi$ is $C^{0,\alpha}(\phi^*T\overline M)$,
the Banach space of all $C^{0,\alpha}$-H\"older vector fields along $\phi$,
endowed with the topology $C^{2,\alpha}$ on the base and $C^{0,\alpha}$ on the fibers;

\item the inclusion $TX\subset\mathcal E$ is obvious, and the inner product on the fibers
$\mathcal E_x$ is induced by the $L^2$-pairing taken with respect to the
volume form (or density) of a fixed background Riemannian metric $\mathbf g_*$ on $M$;

\item $\delta f(\phi,\mathbf g)=-\xi_\mathbf g\,\Delta_{\mathbf g,\overline{\mathbf g}}(\phi)$, where
$\xi_\mathbf g\colon M\to\R$ is the positive $C^{k}$ function satisfying $\xi_\mathbf g\,\mathrm{vol}_{\mathbf g_*}=\mathrm{vol}_{\mathbf g}$, see \eqref{eq:deflaplacian} and \eqref{eq:primavarharm};

\item $\mathcal Y=TX$, $\kappa$ is the identity map and $j$ is induced by the $L^2$-pairing, as above;

\item identifying the Lie algebra $\mathfrak g_\lambda$ of $G_\lambda$ with the space of pairs $(K,\overline K)$ of (complete) Killing vector fields on $(M,\mathbf g)$ and $(\overline M,\overline{\mathbf g})$, for each $\phi\in X$ the map $\dd\beta_\phi^\lambda(1_\lambda)\colon\mathfrak g_\lambda\to T_\phi X$
associates to a pair $(K,\overline K)$ the vector field $\overline{K}\circ\phi-\phi_*(K)$ along $\phi$.
The image of this map is precisely the space of Killing-Jacobi fields along $\phi$, and the notion of equivariant nondegeneracy for harmonic maps
coincides with the abstract notion of nondegeneracy given in our equivariant implicit function theorem;

\item given a $(\mathbf g,\overline{\mathbf g})$-harmonic map $\phi\colon M\to\overline M$, the vertical projection
of the derivative $\partial_1(\delta f)(\phi,\mathbf g)$ is identified via \eqref{eq:ddfjacobi} with $\xi_\mathbf g\, J_\phi$, where $J_\phi$ is the
Jacobi operator in \eqref{eq:harmjacobiop}. This is an elliptic second order partial differential operator and
$\xi_\mathbf g\, J_\phi\colon C^{2,\alpha}(\phi^*T\overline M)\to C^{0,\alpha}(\phi^*T\overline M)$ is a
Fredholm operator of index zero, see \cite[Thm~1.1, (2)]{Whi2}.\qedhere
\end{itemize}
\end{proof}

\subsection{Deformation of Hamiltonian stationary Lagrangian submanifolds}
Let $(M,\omega)$ be a symplectic manifold with $\dim M=2n$ and $\Sigma$ be a compact manifold with $\dim\Sigma=n$. An embedding $x\colon\Sigma\to (M,\omega)$ is called \emph{Lagrangian} if $x^*\omega=0$. In this case, we say $x(\Sigma)\subset M$ is a Lagrangian submanifold.
A smooth family $x_s\colon\Sigma\to M$, $s\in\left(-\varepsilon,\varepsilon\right)$, of embeddings is called a \emph{Hamiltonian deformation} of $x_0=x$ if its derivative $X=\frac{\dd}{\dd s}x_s\big|_{s=0}$ is a Hamiltonian vector field along $x$, i.e., if the $1$-form $\sigma_X:=x^*(\omega(X,\cdot))$ on $\Sigma$ is exact.

Endow $(M,\omega)$ with a Riemannian metric $\mathbf g$. This metric allows us to compute the volume of an embedding $x\colon\Sigma\to M$, by pulling back its volume form $\mathrm{vol}_{\mathbf g}$.
A Lagrangian embedding $x\colon\Sigma\to M$ is called \emph{$\mathbf g$-Hamiltonian stationary} if it has critical volume with respect to any  Hamiltonian deformations $x_s\colon\Sigma\to M$, $s\in\left(-\varepsilon,\varepsilon\right)$, of $x$.
Hamiltonian stationary Lagrangian submanifolds are analogous to minimal submanifolds (i.e., the case $H=0$ in the CMC problem described in Subsection~\ref{subsec:cmc}), but instead of minimizing volume in \emph{all} directions, they minimize volume only in the \emph{Hamiltonian} directions. In particular, minimal Lagrangian submanifolds are automatically Hamiltonian stationary.

We now describe the basis for the variational setup of this problem, following closely \cite[Sec 4]{BetPicSan12}. Let $\lagsigma$ be the space of $C^{4,\alpha}$-unparametrized Lagrangian embeddings of $\Sigma$ in $(M,\omega)$, i.e., the quotient of the space of Lagrangian embeddings of $\Sigma$ into $M$ of class $C^{4,\alpha}$ by the action of the diffeomorphism group of $\Sigma$. The space $\lagsigma$ can be identified with the set of Lagrangian submanifolds of $M$ of class $C^{4,\alpha}$ that are diffeomorphic to $\Sigma$. Its structure is analogous to the set of unparametrized embeddings used in Subsection~\ref{subsec:cmc} to study CMC hypersurfaces, in that, given a smooth Lagrangian embedding $x_0\colon \Sigma\to M$, there exists a local chart around the unparametrized embedding $[x_0]\in\lagsigma$ with values in the Banach space $Z^1(\Sigma)$ of closed $1$-forms on $\Sigma$ of class $C^{4,\alpha}$. In particular, we get an identification
\begin{equation}\label{eq:identclosed}
T_{[x_0]}\lagsigma =Z^1(\Sigma).
\end{equation}
An important finite codimension subspace of $Z^1(\Sigma)$ is the space $B^1(\Sigma)$ of exact $1$-forms on $\Sigma$, of the same regularity. The corresponding distribution, via the identification \eqref{eq:identclosed}, is an integrable distribution of $\lagsigma$ with codimension $b_1(\Sigma)=\dim H^1(\Sigma,\R)$. This distribution is called \emph{Hamiltonian distribution}, and its integral leaves near $[x_0]$ are parametrized by elements of the first de Rham cohomology $H^1(\Sigma,\R)=Z^1(\Sigma)/B^1(\Sigma)$. Denote by $[\eta]\in H^1(\Sigma,\R)$ the cohomology class of a closed $1$-form $\eta$, and by $\lagsigma_{[\eta]}$ the integral leaf of the Hamiltonian distribution corresponding to $[\eta]$. Also, note that $\lagsigma_{[0]}$ is the space of Hamiltonian deformations of $[x_0]$.

Given a smooth unparametrized Lagrangian embedding $[x_0]\in\lagsigma$, let $\mathcal U$ be an open neighborhood of $[x_0]$, that is identified with a neighborhood of $0\in Z^1(\Sigma)$. Shrinking $\mathcal U$ if necessary, suppose that the natural splitting $Z^1(\Sigma)=B^1(\Sigma)\oplus H^1(\Sigma,\R)$ induces a product structure $\mathcal U=\mathcal U_B\times\mathcal U_H$, where $\mathcal U_B$ is a neighborhood of $0\in B^1(\Sigma)$ and $\mathcal U_H$ is a neighborhood of $[0]\in H^1(\Sigma,\R)$. Under this identification, each $[x]\in\mathcal{U}$ corresponds to a unique pair $(\zeta,[\eta])\in\mathcal U_B\times\mathcal U_H$, and $[x_0]$ corresponds to $(0,[0])$.

Set $X=\mathcal U_B$ and $\Lambda_1=\mathcal U_H$.  Let $\mathbf g_t$ be a family of metrics on $M$ of class $C^k$, $k\geq4$, parametrized by $t\in\Lambda_2$, where $\Lambda_2$ is a manifold. Consider the space of parameters $\Lambda=\Lambda_1\times\Lambda_2$, which is a finite-dimensional manifold formed by pairs $\lambda=([\eta],t)$. Let
\begin{equation}\label{eq:hamiltstfunct}
f\colon X\times\Lambda\to\R, \quad f(\zeta,[\eta],t)=\int_\Sigma x^*(\mathrm{vol}_{\mathbf g_t}),
\end{equation}
where $x$ is a Lagrangian embedding that represents the class $[x]=(\zeta,[\eta])\in\mathcal U=\mathcal U_B\times\mathcal U_H$ and $\mathrm{vol}_{\mathbf g_t}$ is the volume form of $(M,{\mathbf g}_t)$. It is easy to verify that the above is well-defined and, for a given $\lambda=([\eta],t)\in\Lambda$, critical points of $f_\lambda(\zeta)=f(\zeta,\lambda)$ are precisely the $\mathbf g_t$-Hamiltonian stationary Lagrangian submanifolds that belong to $\lagsigma_{[\eta]}$. This is just a somewhat intricate way of (locally) encoding the fact that the $\mathbf g_t$-Hamiltonian stationary Lagrangian submanifolds that belong to the leaf $\lagsigma_{[\eta]}$ of the Hamiltonian distribution are the critical points of the $\mathbf g_t$-volume functional restricted to $\lagsigma_{[\eta]}$. More precisely, for a Hamiltonian variation $X$, with $\sigma_X=\dd h$, of the embedding $[x]=(\zeta,[\eta])$,
\begin{equation}\label{eq:firstvarhamil}
\partial_1 f(\zeta,[\eta],t)h=\int_\Sigma \mathbf g_t(\vec H,X)=\int_\Sigma\mathbf g_t(\sigma_H,\dd h)=\int_\Sigma (\dd^*\sigma_H) h,
\end{equation}
where $\vec H$ is the mean curvature vector of $x$ and $\dd^*$ is the codifferential.\footnote{i.e., the formal adjoint of the exterior derivative operator $\dd$, with respect to $\mathbf{g}_t$.} Here, we identify the tangent space to $\mathcal U_B$ at $\zeta$ as $B^1(\Sigma)=C^{4,\alpha}(\Sigma)/\R$, i.e., an exact $1$-form $\dd h$ of class $C^{4,\alpha}$ is identified with the function $h$ modulo constants.

Under the extra assumption that $(M,\omega,\mathbf g_{t_0})$ is a K\"ahler manifold, with complex structure $\mathbf J_{t_0}$, the \emph{Jacobi operator} at a $\mathbf g_{t_0}$-Hamiltonian stationary Lagrangian embedding $x=(\zeta,[\eta])$ is given by the following formula (see \cite[(9), p. 1072]{jls}):
\begin{equation}\label{eq:hamiltstjacobi}
J_{\zeta}(h)=\Delta^2 h+\dd^* \sigma_{\Ric^\perp(\mathbf J_{t_0} \nabla h)}-2\dd^*\sigma_{B(\mathbf J_{t_0} \vec H,\nabla h)}-\mathbf J_{t_0} \vec H(\mathbf J_{t_0} \vec H(h)),
\end{equation}
where $B$ is the second fundamental form of $x(\Sigma)\subset M$ and $\Ric^\perp(X)$ for a normal vector $X$ to $\Sigma$ is defined by $\mathbf g_{t_0}(\Ric^\perp(X),Y)=\Ric(X,Y)$ for all $Y$ normal to $\Sigma$.
Again, the actual linear operator that represents the second variation of the above functional is defined in the tangent space to $\mathcal U_B$, which is $C^{4,\alpha}(\Sigma)/\R$. In other words, \eqref{eq:hamiltstjacobi} induces an operator $[J_{\zeta}]\colon C^{4,\alpha}(\Sigma)/\R \to C^{0,\alpha}(\Sigma)/\R$, since it acts on $\dd h$ and we are representing it by its action on $h$. Exact $1$-forms $\eta$ on $\Sigma$ that are in the kernel of the Jacobi operator at $x$ are called \emph{Jacobi fields}  along $x$.

Assume that $G_t=\mathrm{Iso}(M,\mathbf g_t)$ acts on $(M,\omega)$ in a Hamiltonian fashion, see \cite[Sec 2]{BetPicSan12}. As in the above cases, the functional $f$ is invariant under this action and there are special Jacobi fields that come from the symmetry group. Namely, if $K$ is a Killing field on $(M,\mathbf g_t)$, then the $1$-form $\sigma_K=x^*(\omega(K,\cdot))$ is exact, because the action is Hamiltonian. Moreover, since the action is by symplectomorphisms and $\mathbf g_t$-isometries, $\sigma_K$ is a Jacobi field.

\begin{defin}
Jacobi fields of the form $\sigma_K$ as above, where $K$ is a Killing field, are called \emph{Killing-Jacobi fields}. The $\mathbf g_t$-Hamiltonian stationary Lagrangian embedding $x$ is called \emph{equivariantly nondegenerate} if the space of Jacobi fields along $x$ coincides with the space of Killing-Jacobi fields along $x$.
\end{defin}

\begin{defin}
Two Hamiltonian stationary Lagrangian embeddings $x_i\colon \Sigma\to M$ are said to be \emph{congruent} if there exists a diffeomorphism $\phi$ of $\Sigma$ and an isometry $F\in G_t$, such that $x_2=F\circ x_1\circ\phi$.
\end{defin}

\begin{teo}\label{thm:hamstatlag}
Let $(M,\omega)$ be a symplectic manifold and $\mathbf g_t$ be a family of Riemannian metrics on $M$, parametrized by $t\in\Lambda_2$, for some manifold $\Lambda_2$. Assume that $G_t:=\mathrm{Iso}(M,\mathbf g_t)$ forms a smooth bundle of Lie groups $\mathcal G=\{G_t:t\in\Lambda_2\}$, and that the $G_t$-action on $(M,\omega)$ is Hamiltonian. Suppose $\mathbf g_{t_0}$ is a K\"ahler metric and $x_0\colon\Sigma\to M$ is an equivariantly nondegenerate $\mathbf g_{t_0}$-Hamiltonian stationary Lagrangian embedding. Then, there exists a neighborhood $V_1$ of $[0]\in H^1(\Sigma,\R)$, a neighborhood $V_2$ of $t_0\in \Lambda_2$, a neighborhood $W$ of $[x_0]$ in $\lagsigma$ and a smooth map $x\colon V\to W$, from the neighborhood $V=V_1\times V_2$ of $([0],t_0)$ to $W$, satisfying:
\begin{itemize}
\item[(a)] $x([\eta],t)$ is a $\mathbf g_t$-Hamiltonian stationary Lagrangian embedding which is in $\lagsigma_{[\eta]}$, for all $([\eta],t) \in V$;
\item[(b)] $x([0],t_0)=x_0$.
\end{itemize}
Moreover, given $([\eta],t)\in V$, any other $\mathbf g_t$-Hamiltonian stationary Lagrangian embedding in $\lagsigma_{[\eta]}$ sufficiently close to $x_0$ is congruent to $x(\lambda)$.
\end{teo}

\begin{proof}
Similarly to Theorem~\ref{thm:applCMC}, this is an application of the non-regular version of the equivariant implicit function theorem discussed in Section~\ref{sec:nonregular}. As above, consider a local chart $\mathcal U=\mathcal U_B\times\mathcal U_H\subset B^1(\Sigma)\times H^1(\Sigma,\R)$ of $[x_0]\in\lagsigma$, where $[x_0]$ corresponds to $(0,[0])$, and set the manifold $X$ to be the neighborhood $\mathcal U_B$ of the origin in the space of exact $1$-forms on $\Sigma$ of class $C^{4,\alpha}$. The space of parameters is $\Lambda=\mathcal U_H\times\Lambda_2$, and the functional $f$ is given by \eqref{eq:hamiltstfunct}. Critical points of $f_\lambda\colon X\to\R$, $\lambda=([\eta],t)$, are the (classes of) $g_t$-Hamiltonian stationary Lagrangian embeddings that belong to $\lagsigma_{[\eta]}$.

Analogously to the CMC setup, the map $\beta_{\zeta}^t\colon G_t\to X$ is smooth when $[x]=(\zeta,[\eta])$ is in the dense subset of (classes of) smooth embeddings. The image of $\dd\beta_{\zeta}^t(1_t)$ is the space of Killing-Jacobi fields, so that the equivariant nondegeneracy assumption on $[x]$ is precisely the equivariant nondegeneracy assumption of our implicit function theorem.

The remaining objects used to apply the abstract implicit function theorem are:
\begin{itemize}	
\item $\mathcal E=C^{0,\alpha}(\Sigma)/\R$ is the Banach space of $C^{0,\alpha}$ functions on $\Sigma$ modulo constants. The tangent space to $T_\zeta \mathcal U_B$ is canonically identified with the space of exact $1$-forms on $\Sigma$ of class $C^{4,\alpha}$, hence the inclusion $TX\subset\mathcal E$ is by first identifying such an exact form $\dd h$ with the $C^{4,\alpha}$ function $h$, and then using the inclusion $C^{4,\alpha}\hookrightarrow C^{0,\alpha}$;
\item The pairing $\langle\cdot,\cdot\rangle$ on $\mathcal E$ is the natural $L^2$-pairing of functions, $\langle h_1,h_2\rangle=\int_\Sigma h_1\,h_2\vol_{\mathbf{g}_*}$, with respect to the volume form (or density) of a fixed background Riemannian metric $\mathbf g_*$;
\item $\mathcal Y=C^{3,\alpha}(\Sigma)/\R$, the bundle morphism $\kappa\colon TX\to \mathcal Y$ is the obvious inclusion and $j\colon\mathcal E\to \mathcal Y^*$ is induced by the $L^2$-pairing above;
\item Given $\lambda=([\eta],t)$, identifying the Lie algebra $\mathfrak g_t$ of $G_t$ with the space of (complete) Kiling vector fields on $(M,\mathbf g_t)$, when $\zeta\in\mathcal U_B$ is such that $[x]=(\zeta,[\eta])$ is smooth, then $\dd\beta_\zeta^t(1_t)\colon\mathfrak g_t \to T_\zeta \mathcal U_B$ associates to a Killing field $K$ the exact $1$-form $\sigma_K$;
\item $\delta f_\lambda(\zeta)=\xi_{\mathbf{g}_t} \dd^*\sigma_H$, where $\sigma_H=x^*(\omega(H,\cdot))$ and $x$ is an embedding in the class of $[x]=(\zeta,[\eta])$, if $\lambda=([\eta],t)$ and $\xi_{\mathbf{g}_t}$ is the positive function satisfying $\xi_{\mathbf{g}_t}\vol_{\mathbf{g}_*}=\vol_{\mathbf{g}_t}$, see \eqref{eq:firstvarhamil};
\item the vertical derivative of $\delta f_\lambda$ at $\zeta$ is identified via \eqref{eq:ddfjacobi} with the Jacobi operator $[J_\zeta]$ induced by \eqref{eq:hamiltstjacobi}. By this formula, $J_\zeta$ is a linear elliptic operator of fourth-order, hence Fredholm of index $0$ from $C^{4,\alpha}(\Sigma)$ to $C^{0,\alpha}(\Sigma)$. The induced operator modulo constants, $[J_{\zeta}]\colon C^{4,\alpha}(\Sigma)/\R \to C^{0,\alpha}(\Sigma)/\R$, remains Fredholm of index $0$, see \cite[Sec 4.4]{BetPicSan12}.\qedhere
\end{itemize}
\end{proof}
\end{section}

\appendix
\section{Symmetric Fredholm operators}\label{app}

Consider the following setup:
\begin{itemize}
\item $X$ and $Y$ are Banach spaces;
\item $H$ is a Hilbert space, with inner product $\langle\cdot,\cdot\rangle$;
\item there exist continuous injections $X\hookrightarrow Y\hookrightarrow H$ with dense images;
\item $T\colon X\to Y$ is a (bounded) Fredholm operator;
\item $T$ is symmetric, i.e., $\langle Tx_1,x_2\rangle=\langle x_1,Tx_2\rangle$ for all $x_1,x_2\in X$.
\end{itemize}
It is natural to conjecture that, with the above hypotheses, the Fredholm index of $T$ is zero.
However, this is \emph{false} in general, as shown by the following counter-example.
\begin{exampleA}
Set $H=L^2([0,1])\oplus\mathds R$, and $Y=H$.
Let $u\colon\left]0,1\right]\to\mathds R$ be the function $u(x)=\frac1x$, and let $D\subset L^2([0,1[)$ be the
dense subspace given by $D=\{f\in L^2([0,1]):fu\in L^2([0,1])\}$, endowed with the inner product $\langle f,g\rangle_* = \langle f,g\rangle_{L^2} + \langle fu,gu\rangle_{L^2}$. In other words, this is the inner product that makes $D$ isometric to the graph of the
linear operator $D\ni f\mapsto fu\in L^2([0,1])$.
This is a closed operator, hence $D$ is a Hilbert space when endowed with $\langle\cdot,\cdot\rangle_*$.

Let $\alpha\colon D\to\mathds R$ be the bounded linear functional $\alpha(f)=\int_0^1fu$. Clearly, $\alpha$ is continuous in the topology induced by $\langle\cdot,\cdot\rangle_*$, but discontinuous
in the topology of $H$. Let $X\subset D\oplus\mathds R$ be the graph of $\alpha$. This is a closed subspace
of $D\oplus\mathds R$, when $D$ is endowed with the inner product $\langle\cdot,\cdot\rangle_*$,
and therefore it is a Hilbert space. Moreover, $X$ is dense in $D\oplus\mathds R$ (hence in $Y=H$)
when $D$ is endowed with the topology induced by $H$,
because it is the kernel of the unbounded functional $D\oplus\mathds R\ni(f,t)\mapsto t-\alpha(f)\in\mathds R$.

Let $S\colon D\oplus\mathds R\to H$ the linear map defined by $S(f,c)=(fu,0)$. Then, $S$ is symmetric relatively to the inner product of $H$, and continuous when $D$ is endowed with the topology induced by $\langle\cdot,\cdot\rangle_*$. The kernel of $S$ is $\{0\}\oplus\R$, and the image of $S$ is $L^2\big([0,1]\big)\oplus\{0\}$. Let $T\colon X\to H$ be the restriction of $S$ to $X$. Also
$T$ is symmetric relatively to the inner product of $H$ and continuous in the topology induced by $\langle\cdot,\cdot\rangle_*$.
The kernel of $T$ is $X\cap(\{0\}\oplus\mathds R)=\{0\}$, and the image of $T$ is again $L^2\big([0,1]\big)\oplus\{0\}$, which
is a closed subspace of $H$ having codimension $1$. Thus, $T$ is a Fredholm operator of index $-1$.
\end{exampleA}
Let us show that the above ``conjecture'' is true, provided that additional assumptions be made.
Consider the adjoints of the continuous inclusions $X\hookrightarrow Y\hookrightarrow H$,
which provide continuous inclusions of the dual spaces: $H^*\hookrightarrow Y^*\hookrightarrow X^*$.
Identifying $H$ with $H^*$, we have the chain of inclusions
\[X\hookrightarrow Y\hookrightarrow H=H^*\hookrightarrow Y^*\hookrightarrow X^*.\]
Due to these inclusions and the symmetry of $T$, the adjoint operator $T^*\colon Y^*\to X^*$ is an extension
of $T$.

\begin{propA}
Let $X$, $Y$, $H$, and $T\colon X\to Y$ be as above. In addition, assume
 that the following \emph{ellipticity} hypothesis is satisfied:
\begin{equation}\label{eq:addhyp}
\mbox{For all } x\in Y^*,\ T^*(x)\in Y\ \text{implies}\ x\in X,
\end{equation}
i.e., $(T^*)^{-1}(Y)\subset X$. Then, the Fredholm index of $T$ is zero.
\end{propA}
\begin{proof}
It suffices to show that $Y$ is the direct sum of $\mathrm{Im}(T)$ and $\ker(T)$, where $\mathrm{Im}(T)$ is the image of $T$.
First, the symmetry of $T$ implies that $\mathrm{Im}(T)$ and $\ker(T)$ are orthogonal, hence $\mathrm{Im}(T)\cap\ker(T)=\{0\}$. Given any $y\in Y$, let us show that $y\in\ker(T)+\mathrm{Im}(T)$. Let $V$ denote the closure of $\mathrm{Im}(T)$ in $H$,
and let $z$ be the orthogonal projection of $y$ onto $V$. Then, $y-z\in H$, it is orthogonal to $\mathrm{Im}(T)$, thus
$T^*(y-z)=0$. By hypothesis \eqref{eq:addhyp}, $y-z\in X$ and $T(y-z)=0$, therefore  $y-z\in\ker(T)$.
It remains to show that $z\in\mathrm{Im}(T)$. Since $y\in Y$ and $y-z\in X$, then $z\in Y$.
Moreover, $z\in V$, therefore $z$ is orthogonal to $\ker(T)$, i.e., the linear functional $\langle z,\cdot\rangle$ annihilates $\ker(T)$.
Since $\mathrm{Im}(T)$ is closed, the annihilator of $\ker(T)$ is precisely the image of $T^*$, and therefore there exists $w\in Y^*$ such that $T^*(w)=z$.
But $z\in Y$, and by \eqref{eq:addhyp}, $w\in X$, hence $z\in\mathrm{Im}(T)$.
\end{proof}

We conclude with a remark on the interpretation of the ``ellipticity'' assumption \eqref{eq:addhyp}. In concrete examples, $H$ is an $L^2$-space, $T$ is a formally self-adjoint differential operator between Banach spaces $X$ and $Y$, of functions having a certain regularity, and $T^*$ will be the extension of $T$ to some space of distributions. Given $y\in Y$, \emph{weak solutions} of the equality $T(x)=y$ are vectors $x\in Y^*$ satisfying $T^*(x)=y$. Assumption \eqref{eq:addhyp} means that, when $y\in Y$, weak solutions $x\in Y^*$ of the equality $T(x)=y$ also belong to $X$;
hence are are in fact \emph{true solutions} of the equation. This hypoellipticity property is satisfied, e.g., by linear elliptic differential operators with the appropriate choices of functions spaces $X$ and $Y$.

\end{document}